\documentclass[a4paper,12pt]{amsart}

\usepackage[headings]{fullpage}

\makeindex

\usepackage{amsfonts,graphics,amsmath,amsthm,amsfonts,amscd,amssymb,amsmath,latexsym,euscript, enumerate, mathtools}
\usepackage{epsfig,url}
\usepackage{flafter}
\usepackage[all,cmtip,line]{xy}
\usepackage{array}
\usepackage[english]{babel}
\usepackage{overpic}
\usepackage{subfig}
\usepackage{multirow}
\usepackage{microtype}
\usepackage{hyperref}
\usepackage{marginnote}

\usepackage{wrapfig}
\usepackage[shortlabels]{enumitem}

\usepackage[dvipsnames,svgnames,table]{xcolor}
\usepackage{graphicx}

\newtheorem{theorem}{Theorem}[section]
\newtheorem{lemma}[theorem]{Lemma}
\newtheorem{proposition}[theorem]{Proposition}

\newtheorem{corollary}[theorem]{Corollary}

\theoremstyle{definition} 
\newtheorem{definition}[theorem]{Definition}

\newtheorem{remark}[theorem]{Remark}
\newtheorem{definition-theorem}[theorem]{Definition-Theorem}
\newtheorem{definition-lemma}[theorem]{Definition-Lemma}

\numberwithin{equation}{section}

\newcommand{\C}{\mathbb{C}}
\newcommand{\R}{\mathbb{R}}
\newcommand{\Z}{\mathbb{Z}}

\newcommand{\Q}{\mathbb{Q}}
\newcommand{\mc}{\mathcal}

\newcommand{\mf}{\mathfrak}

\DeclareMathOperator{\ord}{ord}

\DeclareMathOperator{\NE}{NE}

\DeclareMathOperator{\Val}{Val}
\DeclareMathOperator{\QM}{QM}

\DeclareMathOperator{\Spec}{Spec}

\DeclareMathOperator{\Supp}{Supp}

\DeclareMathOperator{\lct}{lct}

\def\mult{\operatorname{mult}}

\newcommand{\stkout}[1]{\ifmmode\text{\sout{\ensuremath{#1}}}\else\sout{#1}\fi}

\usepackage{mathtools}

\DeclarePairedDelimiterX{\norm}[1]{\lVert}{\rVert}{#1}
\newcommand{\floor}[1]{\left\lfloor #1 \right\rfloor}
\newcommand{\ceil}[1]{\left\lceil #1 \right \rceil}
\let\oldframe\frame
\renewcommand\frame[1][allowframebreaks]{\oldframe[#1]}
\title[A valuative approach to the anticanonical minimal model program
]{A valuative approach to the anticanonical minimal model program }

\begin{document}

\author[S.~Choi]{Sung Rak Choi}
\author[S. Jang]{Sungwook Jang}
\author[D.~Kim]{Donghyeon Kim}
\author[D.-W.~Lee]{Dae-Won Lee}
\address[Sung Rak Choi]{Department of Mathematics, Yonsei University, 50 Yonsei-ro, Seodaemun-gu, Seoul 03722, Republic of Korea}
\email{sungrakc@yonsei.ac.kr}
\address[Sungwook Jang]{Center for Complex Geometry, Institute for Basic Science, 34126 Daejeon, Republic of Korea}
\email{swjang@ibs.re.kr}
\address[Donghyeon Kim]{Department of Mathematics, Yonsei University, 50 Yonsei-ro, Seodaemun-gu, Seoul 03722, Republic of Korea}
\email{narimial0@gmail.com}
\address[Dae-Won Lee]{School of Mathematics, Korea Institute for Advanced Study, 85 Hoegiro, Dongdaemun-gu, Seoul 02455, Republic of Korea}
\email{daewonlee@kias.re.kr}


\begin{abstract}
In this paper, we show that the log canonical threshold of a potentially klt triple can be computed by a quasi-monomial valuation.  
As an application of the main result, we show that we can run the MMP on any potentially klt triples and $-(K_X+\Delta)$-MMP on the potentially klt pairs.
\end{abstract}

\maketitle
\allowdisplaybreaks

\section{Introduction}\label{sect:intro}

Throughout the paper, we work over an algebraically closed field of characteristic zero. It is widely acknowledged that a thorough understanding of the geometry of singularities in algebraic varieties has played a crucial role in the development of the minimal model program; conversely, advances in the minimal model program have also substantially deepened our understanding of the singularities. 

The log canonical threshold is a fundamental invariant in the theory of singularities in birational geometry. It is defined as a means to measure the singularities of pairs and it is well known that the behavior of log canonical thresholds has deep connections to the birational geometry of algebraic varieties. For a pair $(X,\Delta)$ and a graded sequence of ideals $\mathfrak{a}_\bullet$, the log canonical threshold $\lct(X,\Delta,\mathfrak{a}_\bullet)$ is defined as
$$
\lct(X,\Delta,\mathfrak{a}_\bullet)=\inf_{\nu\in\Val^*_X}\frac{A_{X,\Delta}(\nu)}{\nu(\mathfrak{a}_\bullet)}
$$
where $\inf$ is taken over all the non-trivial valuations $\nu$ on $X$.
It is believed that $\lct(X,\Delta,\mathfrak{a}_\bullet)$ is computed by some quasi-monomial valuation $\omega\in\Val^*_X$ (Weak Conjecture) and furthermore  $\lct(X,\Delta,\mathfrak{a}_\bullet)$ can be computed only by quasi-monomial valuations (Strong Conjecture). Xu in \cite{Xu2} proves that the Weak Conjecture holds. See \cite{JM} for more details and background on Weak/Strong Conjectures. 

In this paper, we prove the computability of the log canonical thresholds by quasi-monomial valuations for a wider class of objects which we call \textit{potential triples}. A potential triple $(X,\Delta,D)$ consists of a pair $(X,\Delta)$ and a pseudoeffective $\R$-Cartier divisor $D$ on $X$. Such a triple is said to be \emph{potentially klt (pklt)} if $\inf_E\{A_{X,\Delta}(E)-\sigma_E(D)\}>0$, where $\inf$ is taken over all prime divisors $E$ over $X$. For the definitions of the log discrepancy $A_{X,\Delta}(E)$ and the asymptotic divisorial valuation $\sigma_E(D)$, see Section \ref{sect:prelim}. 
These notions provide a framework which allows us to consider further generalizations of the notions and results for pairs and generalized pairs by Birkar--Zhang \cite{BZ}. Indeed, many of the properties satisfied by (generalized) pairs still hold analogously for potential triples although the notion of potential triples is more general than that of generalized pairs (cf. \cite{CP},\cite{CJ},\cite{CJL2}). 
We also note that the notion of triples has already appeared in various contexts, whether implicitly or not, in \cite{JM}, \cite{Xu2}, \cite{Xu3}, \cite{Xu4}, etc., where the properties of pairs $(X,\Delta)$ together with graded linear systems $\mathfrak{b}_\bullet$ are studied in relation with valuations theory and $K$-stability. Recently, an interesting characterization of Fano type varieties in terms of potential triples was obtained  \cite{CJK} and we also expect that the minimal model program on potential triples can provide a key to establishing the $-K_X$-MMP. A condition on a big divisor $L$ with respect to a pair $(X,\Delta)$ which is equivalent to being pklt was also defined in \cite[Lemma 3.8]{HQZ} for the study of local volumes.
See Section \ref{sect:prelim} for precise definitions of potential triples and related notions.

The following is the first main theorem of the paper. The log canonical threshold of a potential triple $(X,\Delta,D)$ is defined as 
$$
\lct_{\sigma}(X,\Delta,D)=\inf_{\nu\in\Val^*_X}\frac{A_{X,\Delta}(\nu)}{\sigma_{\nu}(D)}
$$
where $\inf$ is taken over all non-trivial valuations $\nu$ on $X$.
For the definition of asymptotic valuation $\nu_{\sigma}(D)$, see Subsection \ref{subsec:asympt val}.

\begin{theorem} \label{thrm:1}
Let $(X,\Delta,D)$ be a pklt triple. Then there exists a quasi-monomial valuation $\omega$ over $X$ which computes $\lct_{\sigma}(X,\Delta,D)$, i.e., $\lct_{\sigma}(X,\Delta,D)=\frac{A_{X,\Delta}(\omega)}{\sigma_{\omega}(D)}$.
\end{theorem}

We note that Theorem \ref{thrm:1} extends the main result of \cite{Xu2} to the case where $D$ is not necessarily big. 
We also remark that Theorem \ref{thrm:1} gives a partial affirmative answer to \cite[Question 6.15]{Leh}.
As a corollary of Theorem \ref{thrm:1}, we prove in Proposition \ref{prop:perturbation} that for a given pklt triple $(X,\Delta,D)$, the triple $(X,\Delta,(1+\epsilon)D)$ is also pklt for a sufficiently small $\epsilon>0$. In turn, this perturbation gives us the following useful applications.

\begin{theorem}\label{thrm:2}
Let $(X,\Delta,D)$ be a pklt triple. Then, for any $0< \varepsilon\ll 1$, we can run the $(K_X+\Delta+(1+\varepsilon)D)$-MMP with scaling of an ample divisor.
\end{theorem}

Using Theorem \ref{thrm:2}, we can easily obtain the following anticanonical analog of the main result of \cite{BCHM}.

\begin{corollary} \label{cor:1}
Let $(X,\Delta,-(K_X+\Delta))$ be a pklt triple. Then we can run the $-(K_X+\Delta)$-MMP with scaling of an ample divisor.
\end{corollary}

Note that in Theorem \ref{thrm:2} and Corollary \ref{cor:1}, we may assume that the respective MMP is eventually a sequence of flips, whose termination is still unknown. Thus, in the surface case, we obtain a $-(K_X+\Delta)$-minimal model of $(X,\Delta)$ since there are no flips in dimension $2$.

\begin{corollary} \label{cor:2}
Let $(X,\Delta,-(K_X+\Delta))$ be a pklt triple of dimension $2$. Then there exists a $-(K_X+\Delta)$-minimal model of $(X,\Delta)$.
\end{corollary}

\medskip

We explain the strategy of the proof of Theorem \ref{thrm:1}. Suppose that we are given a pklt triple $(X,\Delta,D)$. We first show that there exists a sequence of valuations $\nu_{\ell}$ that approximates the log canonical threshold $\lct_{\sigma}\left(X,\Delta,D+\frac{1}{\ell}A\right)$ (in the sense of Proposition \ref{prop:enlarged ideal for lct}) such that all centers of $\nu_{\ell}$ coincide, say, $c_X(\nu_{\ell})=x$. Then we show that there exists some $\delta>0$ such that $\nu_{\ell}(\mathfrak{m}_x)\ge \delta$ for all $\ell$. In the course of the proof, we use the theory of diminished multiplier ideals introduced in \cite{Hac,Leh}. Finally, by the continuity of $\nu\mapsto \sigma_{\nu}(D)$ on $\mathrm{Val}^{\le M}_X\coloneqq \{\nu\in \mathrm{Val}_X\mid A_{X,\Delta}(\nu)\le M\}$ for every $M>0$, we can prove that there exists a convergent subsequence of $\{\nu_{\ell}\}$ whose limit computes $\lct_{\sigma}(X,\Delta,D)$, and is actually a quasi-monomial valuation.

The rest of this paper is organized as follows. In Section \ref{sect:prelim}, we recall the definition of potential triples and the associated notions such as pklt and plc triples. We also recall the notions of quasi-monomial valuation, asymptotic order, and log canonical threshold. Section \ref{sect:asympt mult ideal} is devoted to collecting notions and lemmas on the asymptotic multiplier ideals. In Sections \ref{sect:asympt mult ideal} and \ref{sect:proofs}, we prove our main results: Theorem \ref{thrm:1} and Theorem \ref{thrm:2}.

\section*{Acknowledgement}
The authors thank professor Chuyu Zhou for his comments and questions.
S. Choi, D. Kim, and D.-W. Lee are partially supported by the Samsung Science and Technology Foundation under Project Number SSTF-BA2302-03. D. Kim is partially supported by Basic Science Research Program through the National Research Foundation of Korea (NRF) funded by the Ministry of Education (No. RS-2026-25555589). D.-W. Lee is also partially supported by Basic Science Research Program through the National Research Foundation of Korea (NRF), funded by the Ministry of Education (No. RS-2023-00237440 and 2021R1A6A1A10039823) and supported by a KIAS Individual Grant (MG111201) at Korea Institute for Advanced Study.

\section{Preliminaries}\label{sect:prelim}
In this paper, we work over an algebraically closed field $K$ of characteristic zero, e.g., the complex number field $\C$. A variety over $K$ is an integral and separated scheme of finite type over $K$, unless otherwise stated. For a normal variety $X$ and a Cartier divisor $D$ on $X$, we denote by $\mathfrak{b}(|D|)$ the ideal of the base locus of $|D|$. For a closed subscheme $N\subseteq X$, $\mathcal{I}_N$ is the ideal sheaf of $\mathcal{O}_X$ corresponding to $N$. For an ideal $\mathfrak{a}\subseteq \mathcal{O}_X$, $\mathcal{Z}(\mathfrak{a})\subseteq X$ is the closed subscheme corresponding to $\mathfrak{a}$. For a scheme $X$ and a point $x\in X$, $\mathfrak{m}_x$ denotes the maximal ideal of $\mathcal{O}_{X,x}$.

\subsection{Pairs and triples}

A \textit{pair} $(X,\Delta)$ consists of a normal variety $X$ and an effective $\Q$-divisor $\Delta$ on $X$ such that $K_X+\Delta$ is $\Q$-Cartier. We say that a proper birational morphism $f\colon Y\to X$ is a \textit{log resolution} of a pair $(X,\Delta)$ if $Y$ is smooth and $\Supp(f^{-1}_*\Delta+\mathrm{Exc}(f))$ is a simple normal crossing divisor. For a flat morphism $X\to V$ from $X$ to a normal variety $V$, $f\colon Y\to X$ is a \textit{fiberwise log resolution} of $(X,\Delta)\to V$ if the following two conditions are satisfied:

\begin{enumerate}[(1)]
    \item for each $u\in V$, $Y_u\to X_u$ is a log resolution of $(X_u,\Delta_u)$, and
    \item if we let $E\coloneqq \Supp(f^{-1}_*\Delta)+\mathrm{Exc}(f)=\sum\limits_{i\in I} E_i$, then $E_J\coloneqq \bigcap\limits_{j\in J}E_j$ has geometric irreducible fibers over $V$.
\end{enumerate}

If $x\in X$ is a point on $X$, then $(X_x\coloneqq \Spec \mathcal{O}_{X,x},\Delta_x)$ denotes the localization of $X$ and $\Delta$ at $x$. 

Let $E$ be a prime divisor over $X$, i.e., there exists a proper birational morphism $f\colon Y\to X$ from a normal projective variety $Y$ such that $E$ is a prime divisor on $Y$. For such $f$, there exists a divisor $\Delta_Y$ on $Y$ such that $K_Y+\Delta_Y=f^*(K_X+\Delta)$. The \emph{log discrepancy} $A_{X,\Delta}(E)$ of $E$ with respect to $(X,\Delta)$ is defined as $A_{X,\Delta}(E)\coloneqq  1-\mult_E(\Delta_Y)$, where $\mult_E(\Delta_Y)$ denotes the coefficient of $E$ in $\Delta_Y$. It is well known that $A_{X,\Delta}(E)$ depends only on the divisorial valuation on $K(X)$ defined by $E$. A pair $(X,\Delta)$ is said to be \textit{kawamata log terminal (or klt)} if $\inf\limits_{E}A_{X,\Delta}(E)>0$ holds where $\inf$ is taken over all the prime divisors $E$ over $X$. If the inequality $\inf\limits_{E}A_{X,\Delta}(E)\geq 0$ holds, then the pair $(X,\Delta)$ is said to be \textit{log canonical} (or \textit{lc}). We note that due to \cite[Corollary 2.31]{KM}, these definitions are equivalent to the traditional definitions that are given with the conditions $A_{X,\Delta}(E)>0$ and $A_{X,\Delta}(E)\geq 0$, respectively, for all prime divisors $E$ over $X$. A pair $(X,\Delta)$ is said to be \emph{purely log terminal (plt)} if $A_{X,\Delta}(E)>0$ for every exceptional divisor $E$ over $X$.  For more details on the notions related to pairs, see \cite{KM},\cite{Kol}.

A \textit{potential triple} $(X,\Delta,D)$ consists of a pair $(X,\Delta)$ with a normal projective variety $X$ and a pseudoeffective $\Q$-Cartier divisor $D$ on $X$. This notion was first introduced and studied by Choi--Park in \cite{CP} in the case where $D=-(K_X+\Delta)$. Recently, the notion of potential pair was generalized to the current form in \cite{CJK}. Potential triples can be considered as a natural generalization of generalized pairs defined by Birkar-Zhang \cite{BZ}. One can easily see that potential triples satisfy similar properties of (generalized) pairs. See also Remark \ref{rmk:potential triple} and \cite[Remark 2.9]{CJK}.

Let $Y$ be a smooth projective variety and $D$ a big $\Q$-Cartier divisor on $Y$. For a prime divisor $E$ on $Y$, the \textit{asymptotic divisorial valuation} of $D$ along $E$ is defined as $\sigma_E(D)\coloneqq \inf\{\mult_E(D')\mid D'\in |D|_{\Q}\}$.  If $D$ is a pseudoeffective $\Q$-Cartier divisor, then we define $\sigma_E(D)\coloneqq  \lim\limits_{\varepsilon\rightarrow 0^+} \sigma_E(D+\varepsilon A)$, where $A$ is an ample Cartier divisor on $Y$. We note that the limit is independent of the choice of $A$. It is also known that there are only finitely many prime divisors $E$ on $Y$ such that $\sigma_E(D)>0$ by \cite[Corollary III.1.11]{Nak}.

Let $X$ be a normal projective variety and $D$ a pseudoeffective $\Q$-Cartier divisor on $X$. For a prime divisor $E$ over $X$, we define $\sigma_E(D)$ as $\sigma_E(D)\coloneqq \sigma_E(f^\ast D)$ where $f\colon Y\rightarrow X$ is a resolution of $X$ such that $E$ is a prime divisor on $Y$. This definition does not depend on the choice of the resolution \cite[Theorem III.5.16]{Nak}.

\begin{definition}
Let $(X,\Delta,D)$ be a potential triple. For a prime divisor $E$ over $X$, we define the \textit{potential log discrepancy} $a(E;X,\Delta,D)$ of $E$ with respect to $(X,\Delta,D)$ as
$$ a(E;X,\Delta,D)\coloneqq  A_{X,\Delta}(E)-\sigma_E(D).$$
We say that $(X,\Delta,D)$ is \textit{weakly potentially kawamata log terminal (weakly pklt)} if $a(E;X,\Delta,D)>0$ for any prime divisor $E$ over $X$. We say that the potential triple $(X,\Delta,D)$ is \textit{potentially kawamata log terminal (pklt)} if
$$\inf_E a(E;X,\Delta,D)>0$$
where $\inf$ is taken over all prime divisors $E$ over $X$.

If $D=-(K_X+\Delta)$ is pseudoeffective in a potential triple $(X,\Delta,D)$, then we simply call $(X,\Delta)$ a \emph{potential pair}.
\end{definition}

\begin{remark}\label{rmk:potential triple}
Note that if $D$ admits a birational Zariski decomposition, then we can reduce the problems to the case of generalized pairs. In particular,  in a $\Q$-factorial potential triple $(X,\Delta,D)$, if $D$ is the trace of a $b$-nef divisor, then $(X,\Delta+D)$ is simply a generalized pair. Since not all pseudoeffective divisors appear as the image of nef divisors and not all pseudoeffective divisors have birational Zariski decompositions, the notion of potential triples is more general than that of generalized pairs. See \cite{BZ}, \cite{B2},\cite[Remark 2.9]{CJK}, \cite{CHLX} and the references therein for more details.
\end{remark}

Let $X$ be a normal projective variety and $D$ an effective $\Q$-Cartier divisor on $X$. A birational contraction $\varphi\colon X\dashrightarrow Y$ is said to be \textit{$D$-nonpositive} (resp. \textit{$D$-negative}) if $\varphi_*D$ is a $\Q$-Cartier divisor on a normal projective variety $Y$, and there is a common resolution $(p,q)\colon Z\rightarrow X\times Y$ such that
$$
   p^\ast D=q^\ast \varphi_\ast D+E
$$
where $E\geq 0$ is a $q$-exceptional divisor (resp. additionally, $\Supp(E)$ contains all the strict transforms of the $\varphi$-exceptional divisors).

\begin{definition}\hfill
\begin{enumerate}[(1), leftmargin=8mm]
\item Let $(X,\Delta,D)$ be a pklt triple. For a birational contraction $f\colon X\dashrightarrow Y$, let $(Y,\Delta_Y,D_Y)$ be a potential triple where $\Delta_Y\coloneqq f_*\Delta$ and $D_Y\coloneqq f_*D$. Then $(Y,\Delta_Y,D_Y)$ is called a \emph{minimal model of} $(X,\Delta,D)$ if $f$ is $(K_X+\Delta+D)$-nonpositive, $(Y,\Delta_Y,D_Y)$ is pklt, and $K_Y+\Delta_Y+D_Y$ is nef.

\item Let $(X,\Delta)$ be a pklt pair. For a birational contraction $f\colon X\dashrightarrow Y$, let $(Y,\Delta_Y)$ be a potential pair where $\Delta_Y\coloneqq f_*\Delta$. Then $(Y,\Delta_Y)$ is called a $-(K_X+\Delta)$-\emph{minimal model} of $(X,\Delta)$ if $f$ is $-(K_X+\Delta)$-nonpositive, $(Y,\Delta_Y)$ is a pklt pair, and $-(K_Y+\Delta_Y)$ is nef.
\end{enumerate}
\end{definition}

We expect that the above minimal models can be obtained by running the MMP (with scaling) on the potential triples and potential pairs, respectively, which we explain below.

We note that for a pklt triple $(X,\Delta,D)$ with a big $\Q$-divisor $D$, we can associate a klt pair as follows.

\begin{theorem}[cf. {\cite[Theorem 1.1]{CJK}}]\label{thm:compl}
Let $(X,\Delta,D)$ be a potential triple with $D$ a big $\Q$-divisor. Then $(X,\Delta,D)$ is pklt if and only if there is an effective divisor $D'\sim_{\Q}D$ such that $(X,\Delta+D')$ is klt.
\end{theorem}

The following is also well known for klt pairs (i.e., $(X,\Delta)$ is klt and $D=0$).

\begin{proposition}[{\cite[Lemma 2.5]{CJL2}}] \label{prop:pot discr}
    Let $(X,\Delta,D)$ be a potential triple and $R$ a $(K_{X}+\Delta+D)$-negative extremal ray of $\overline{\NE}(X)$. Assume that we have either a divisorial contraction or a flip $\varphi\colon X\dashrightarrow X'$ associated to the ray $R$. Let $\Delta'\coloneqq \varphi_*\Delta$ and $D'\coloneqq \varphi_*D$. Then we have
$$a(E;X,\Delta,D)\leq a(E;X',\Delta',D') $$
for any prime divisor $E$ over $X$.
\end{proposition}

\begin{definition}
Let $(X,\Delta,D)$ be a pklt triple and $A$ an ample divisor on $X$. Let 
$$ (X,\Delta,D)\coloneqq (X_0,\Delta_0,D_0) \overset{f_1}{\dashrightarrow} (X_1,\Delta_1,D_1)\overset{f_2}{\dashrightarrow} \cdots$$
be a sequence of $(K_X+\Delta+D)$-negative extremal divisorial contractions and flips associated to a $(K_{X_i}+\Delta_i+D_i)$-negative extremal ray $R_i$, where $\Delta_i,D_i$ and $A_i$ are the strict transforms of $\Delta,D$ and $A$ on $X_i$, respectively. Denote $\lambda_i\coloneqq  \inf\{t\ge 0\mid K_{X_i}+\Delta_i+D_i+tA_i\text{ is nef}\}$. We say that the sequence is \textit{a $(K_X+\Delta+D)$-MMP with scaling of $A$} if $(K_{X_i}+\Delta_i+D_i+\lambda_iA_i)\cdot R_i=0$.
\end{definition}

\begin{lemma} \label{lem:pklt to plc}
Let $(X,\Delta,D)$ be a potential triple. If $(X,\Delta,(1+\varepsilon)D)$ is plc for some $\varepsilon>0$ and $(X,\Delta,D)$ is pklt, then $(X,\Delta,(1+\varepsilon_{0})D)$ is pklt for any $\varepsilon_0$ such that $0\le \varepsilon_{0}<\varepsilon$.
\end{lemma}
\begin{proof}
Let $\alpha=\min\limits_E\{A_{X,\Delta}(E)-\sigma_E(D)\}$,  where $\min$ is taken over all the prime divisors $E$ over $X$. Then, by assumption, we have $\alpha>0$ and 
$$ A_{X,\Delta}(E)-(1+\varepsilon)\sigma_{E}(D)\ge 0 $$
for all prime divisors $E$ over $X$. Fix a real number $\varepsilon_0$ such that $0\le \varepsilon_{0}<\varepsilon$. If $\sigma_{E}(D)\le \frac{\alpha}{2\varepsilon_{0}}$, then we have
$$ A_{X,\Delta}(E)-(1+\varepsilon_{0})\sigma_{E}(D)\ge \alpha-\frac{\alpha}{2}=\frac{\alpha}{2}. $$
If $\sigma_{E}(D)> \frac{\alpha}{2\varepsilon_{0}}$, then we have
\begin{equation*}
    A_{X,\Delta}(E)-(1+\varepsilon_{0})\sigma_{E}(D)=A_{X,\Delta}(E)-(1+\varepsilon)\sigma_{E}(D)+(\varepsilon-\varepsilon_0)\sigma_E(D)> \frac{\alpha(\varepsilon-\varepsilon_{0})}{2\varepsilon_{0}}.
\end{equation*}
Therefore, $\inf_E\{A_{X,\Delta}(E)-(1+\varepsilon_{0})\sigma_{E}(D)\}\geq \min\left\{\frac{\alpha}{2},\frac{\alpha(\varepsilon-\varepsilon_{0})}{2\varepsilon_{0}}\right\}>0$, which implies that $(X,\Delta,(1+\varepsilon_0)D)$ is pklt.
\end{proof}

\subsection{Valuations} \label{valuation}
Let $X$ be a normal variety. We recall several notions related to the valuations of $X$ mainly from \cite[Section 4.1]{JM}. For a valuation $\nu$ on $K(X)$, if there is a point $x$ of $X$ that induces a local inclusion $\mathcal{O}_{X,x}\hookrightarrow \mathcal{O}_{\nu}$, then such a point $x$ is called the \textit{center} of $\nu$, and we denote it by $c_X(\nu)\coloneqq x$. If the center $c_X(\nu)$ is the generic point of $X$, then $\nu$ is called a \textit{trivial valuation} of $X$.  We denote by $\Val_{X}$ (resp. $\Val^{\ast}_X$) the set of all (resp. nontrivial) valuations on $K(X)$ having centers in $X$. We denote by $\mathrm{Val}_{X,x}\subseteq\mathrm{Val}_{X}$ the subset of all valuations of $X$ whose center is the point $x$ of $X$. 

We consider $\Val_X$ as the topological space equipped with the following topology. Let $\mathcal{I}$ be a set of nonzero ideals on $X$. To a valuation $\nu\in \Val_X$ and a coherent ideal sheaf $\mf{a}$ on $X$, one can associate a function $\nu\colon \mathcal{I}\rightarrow \R_{\geq 0}$ by defining $\nu(\mf{a})\coloneqq \min\{\nu(f) \mid f\in \mf{a}\cdot \mathcal{O}_{X,x}\}$, where $x=c_X(\nu)$. This function can be seen as a homomorphism between semirings and endows a topology $\tau$ on $\Val_X$, i.e., $\tau$ is the weakest topology on $\Val_X$ such that the evaluation map $\nu \to \nu(\mf{a})$ is continuous for all nonzero ideals $\mf{a}$ on $X$. See \cite[Lemma 4.1]{JM} for another characterization of the topology $\tau$ on $\Val_X$.

For two valuations $\nu_1, \nu_2\in \Val_X$, we write $\nu_1\leq \nu_2$ if $\nu_1(\mf{a})\leq \nu_2(\mf{a})$ for all nonzero ideals $\mf{a}$ on $X$.

For a valuation $\nu\in \Val^\ast_{X}$ and $\lambda\in \R_{\ge0}$, we define the \textit{valuation ideal sheaf} $\mf{a}_{\lambda}(\nu)$ by
$$\mf{a}_{\lambda}(\nu)(U)\coloneqq \{f\in \mc{O}_{X}(U) \mid \nu(f)\ge \lambda\} 
$$
for an open subset $U\subseteq X$.
Note that $\mf{a}_{\bullet}(\nu)\coloneqq \{\mf{a}_{m}(\nu)\}_{m\in \Z_{\ge0}}$ is a graded sequence of ideals.

Now we recall the notion of \emph{quasi-monomial valuations}.
Suppose that $f\colon Y\to X$ is a projective birational morphism from a smooth projective variety $Y$, and $E\coloneqq \sum_{i\in I}E_i$ is a simple normal crossing divisor on $Y$. We further assume that $f$ is an isomorphism outside of $\Supp E$. Such a pair $(Y,E)$ is called a \emph{log-smooth pair} over $X$. Let $V$ be an irreducible component $V=\bigcap_{i\in I'}E_i(\neq\varnothing)$ for some $I'\subseteq I$, and $\eta$ the generic point of $V$. By the Cohen's structure theorem (cf. \cite[Theorem 032A]{Stacks}), there exists a system of algebraic coordinates $z_i$ at $\eta$ such that
$ \widehat{\mathcal{O}_{Y,\eta}}\cong \C(\eta)[[z_1,\cdots,z_r]].$
Note that any $f\in\mathcal{O}_{Y,\eta}$ can be written as $f=\sum\limits_{\beta\in \Z^r_{\ge 0}} a_{\beta} z^{\beta}$ for some $a_{\beta}\in\C(\eta)$, which is either zero or a unit. Here, we use the notation $z^{\beta}=z_1^{\beta_1}z_2^{\beta_2}\cdots z_r^{\beta_r}$ for  $\beta=(\beta_1,\cdots,\beta_r)\in\Z^r_{\ge 0}$.
For each $\alpha\coloneqq (\alpha_1,\cdots,\alpha_r)\in \R^r_{\ge 0}$ and each choice of the point $\eta$, we can define a valuation $\nu_{\alpha,\eta}$ on $\widehat{\mathcal{O}_{Y,\eta}}$ as
\begin{align*}
\nu_{\alpha,\eta}\left(f\right)\coloneqq \min\left\{\sum^r_{i=1}\alpha_i\beta_i\biggm|a_{\beta}\ne 0\right\}.
\end{align*}
Any valuation in $\Val_X$ of the form $\nu_{\alpha,\eta}$ for some log-smooth pair $(Y,E)$ over $X$ is called a \emph{quasi-monomial valuation} (at $\eta$ with respect to $\alpha\in\R^r_{\geq0}$). Note that if there is a real number $c>0$ such that $c\alpha\in\Z^r_{\geq0}$, then $\nu_{\alpha,\eta}$ is simply a divisorial valuation. 

We say that the log-smooth pair $(Y,E)$ over $X$ is \emph{adapted to a quasi-monomial valuation $\nu$} if $\nu$ is of the form $\nu_{\alpha,\eta}$ described above. Let $\QM(Y,E)$ be the collection of quasi-monomial valuations adapted to $(Y,E)$. Then, any quasi-monomial valuation of $X$ belongs to $\QM(Y,E)$ for some log-smooth pair $(Y,E)$ over $X$. We fix the topology on $\mathrm{QM}(Y,E)\subseteq \mathrm{Val}_X$ as the subspace topology of $\tau$. It is known that there exists a \emph{retraction map} $r_{(Y,E)}\colon \Val_X\to \QM(Y,E)$ that is continuous and is the identity on $\QM(Y,E)$. 

We next extend the definition of the log discrepancy function $A_{X,\Delta}$ for arbitrary valuations in $\Val_X$ as follows. Let $(Y,E=\sum_{i\in I}E_i)$ be a log-smooth pair of $X$ where $Y$ is a log resolution of $(X,\Delta)$. Such a pair $(Y,E)$ is  called a log-smooth model of $(X,\Delta)$. Let $\nu=\nu_{\alpha,\eta}$ be a quasi-monomial valuation in $\QM(Y,E)$ associated to some $\alpha=(\alpha_{1},\dots,\alpha_{r})\in \R_{\ge0}^{r}$ and a generic point of an irreducible component of $\cap_{j\in I'}E_j$ with $I'\subseteq I$ such that $|I'|=r$. Then the log discrepancy $A_{X,\Delta}(\nu)$ of $\nu$ with respect to $(X,\Delta)$ is defined as
$$A_{X,\Delta}(\nu)\coloneqq \sum_{j\in I'}\alpha_{j}A_{X,\Delta}(E_{j}).$$
We further define the log discrepancy $A_{X,\Delta}(\nu)$ for an arbitrary valuation $\nu\in\mathrm{Val}_X$ with respect to $(X,\Delta)$ as
$$A_{X,\Delta}(\nu)\coloneqq \sup_{(Y,E)}A_{X,\Delta}(r_{(Y,E)}(\nu)),$$
where $r_{(Y,E)}\colon \Val_X\to\QM(Y,E)$ is the retraction map, and the sup is taken over all log-smooth pairs $(Y,E)$ over $X$.

\subsection{Dual complex}\label{dual complex}
In this subsection, we recall the notion of \textit{dual complex} from \cite{dFKX}.

Let $(X,\Delta)$ be a klt pair and $\left(Y,E\coloneqq \sum_{i=1}^{r}E_i\right)$ a log-smooth model of $(X,\Delta)$. 
The \textit{dual complex} $\mathcal{D}(E)$ associated to $E$ is a regular cell complex defined as follows: 
prime divisors $E_i$ correspond to the vertices of the complex $\mathcal{D}(E)$ and each irreducible component $W$ of $\bigcap\limits_{i\in J}E_i (\neq \varnothing)$ for some subset $J=\{i_1,\cdots,i_l\}\subseteq I=\{1,\cdots,r\}$ corresponds to a cell of codimension $|J|-1$ in $\mc{D}(E)$.
Let us write the cell associated to $J$ by 
$$ W_J\coloneqq \left\{(\alpha_{i_1},\cdots,\alpha_{i_{\ell}})\in\R^{\ell}_{>0}\left|\;\sum^{\ell}_{j=1} \alpha_{i_j}=1\right.\right\}.$$

We can identify $\mathcal{D}(E)$ as a subset of $\mathrm{Val}_X$ as follows. For any $\alpha=(\alpha_{i_1},\cdots,\alpha_{i_\ell})\in W_J$, we can define a quasi-monomial valuation $\imath_{X,\Delta}(\alpha)\coloneqq \nu_{\alpha'}$ associated to $\alpha'=(\alpha'_1,\alpha'_2,\cdots,\alpha'_\ell)\in\R^\ell_{\geq0}$ where $\alpha'_{j}=\frac{\alpha_{i_j}}{A_{X,\Delta}(E_{i_j})}$. Then this 
 defines a natural injection $\imath_{X,\Delta}\colon \mathcal{D}(E)\to \mathrm{Val}_{X}$ (cf. \cite{dFKX}). Therefore, we can consider $\mathcal{D}(E)\subseteq \mathrm{Val}_{X}$ as a subset of quasi-monomial valuations. In particular, we can identify $\mathcal{D}(E)$ with the set of quasi-monomial valuations $\nu\in \mathrm{QM}(Y,E)$ with $A_{X,\Delta}(\nu)=1$.

Let $V$ be a smooth variety and $Y$ a smooth scheme that is a fiberwise log resolution of $X\times V\to V$. Let $D'$ be an effective divisor on $X\times V$ and $c>0$. Suppose that $E_1,\cdots,E_r$ are log canonical places of $(X\times V,\Delta\times V+cD')$. Then, for any $u\in V$, the reductions $(E_1)_u,\cdots,(E_r)_u$ give all the prime divisors on $Y_u$ that are log canonical places of $(X,\Delta+cD'_u)$. Moreover, we can identify $\mathcal{D}\left(E\coloneqq \sum\limits_{i=1}^r E_i\right)$ and $\mathcal{D}(E_u)$. For more details on the dual complexes, see \cite{dFKX}.

\subsection{Asymptotic order}\label{subsec:asympt val}
Let $D$ be an effective $\Q$-Cartier divisor and $\mf{a}_{m}=\mf{b}(|mD|)$ the base ideal of $|mD|$. If $mD$ is not Cartier, then we set $\mf{a}_m=(0)$. For a valuation $\nu\in \Val_{X}$, the \emph{asymptotic order of vanishing} $\nu(\|D\|)$ of $D$ along $\nu$ is defined by
$$ 
\nu(\|D\|)\coloneqq  \inf_{m}\frac{\nu(\mf{a}_{m})}{m}. 
$$

For a pseudoeffective $\Q$-Cartier divisor $D$, we define the asymptotic valuation $\sigma_{\nu}(D)$ of $D$ along $\nu$ by
$$\sigma_{\nu}(D)\coloneqq  \lim_{\varepsilon\to 0^{+}}\nu(\|D+\varepsilon A\|),
$$
where $A$ is an ample divisor. Note that the limit does not depend on the choice of $A$ by \cite[Theorem 2]{Jow}. Moreover, it is known that if $D$ is big, then the equality $\sigma_{\nu}(D)=\nu(\|D\|)$ holds.

Next, we prove that the (asymptotic) valuations are invariant under base field extension.

\subsection{Log canonical threshold}
We recall the definition of \emph{log canonical thresholds} and related results.

\begin{definition}\hfill
\begin{enumerate}[(1), leftmargin=8mm]
\item Let $(X,\Delta)$ be a pair and $\mathfrak{a},\mathfrak{q}$ ideals of $\mathcal{O}_X$. The \emph{log canonical threshold $\lct^{\mathfrak{q}}(X,\Delta,\mathfrak{a})$ of $(X,\Delta)$ with respect to $\mathfrak{a}$ relative to $\mathfrak{q}$} is defined as
$$ \lct^{\mathfrak{q}}(X,\Delta,\mathfrak{a})\coloneqq \inf_{\nu\in \mathrm{Val}^*_X}\frac{A_{X,\Delta}(\nu)+\nu(\mathfrak{q})}{\nu(\mathfrak{a})}.$$
We say a valuation $\nu_0\in \mathrm{Val}^*_X$ \emph{computes} $\lct^{\mathfrak{q}}(X,\Delta,\mathfrak{a})$ if
$$ \lct^{\mathfrak{q}}(X,\Delta,\mathfrak{a})=\frac{A_{X,\Delta}(\nu_{0})+\nu_0(\mathfrak{q})}{\nu_0(\mathfrak{a})}.$$
When $\mathfrak{q}=\mathcal{O}_X$, we drop ``$\mathfrak{q}$'' in the notation.

\item Let $(X,\Delta)$ be a pair, $\mathfrak{q}$ an ideal of $\mathcal{O}_X$, and $\mathfrak{a}_{\bullet}$ a graded sequence of ideals of $\mathcal{O}_X$. The \emph{log canonical threshold $\lct^{\mathfrak{q}}(X,\Delta,\mathfrak{a}_{\bullet})$ of $(X,\Delta)$ with respect to $\mathfrak{a}_\bullet$ relative to $\mathfrak{q}$} is defined as
$$\lct^{\mathfrak{q}}(X,\Delta,\mathfrak{a}_{\bullet})\coloneqq  \inf_{\nu\in \mathrm{Val}^*_X}\frac{A_{X,\Delta}(\nu)+\nu(\mathfrak{q})}{\nu(\mathfrak{a}_{\bullet})}.$$
We say a valuation $\nu_0\in \mathrm{Val}^*_X$ \emph{computes} $\lct^{\mathfrak{q}}(X,\Delta,\mathfrak{a}_{\bullet})$ if
$$ \lct^{\mathfrak{q}}(X,\Delta,\mathfrak{a}_{\bullet})=\frac{A_{X,\Delta}(\nu_{0})+\nu_0(\mathfrak{q})}{\nu_0(\mathfrak{a}_{\bullet})}.$$
When $\mathfrak{q}=\mathcal{O}_X$, we drop ``$\mathfrak{q}$'' in the notation.

\item Let $(X,\Delta,D)$ be a potential triple. The \emph{log canonical threshold} $\lct_{\sigma}(X,\Delta,D)$ of $(X,\Delta,D)$ is defined as
$$ \lct_{\sigma}(X,\Delta,D)\coloneqq  \inf_{\nu\in \mathrm{Val}^*_X}\frac{A_{X,\Delta}(\nu)}{\sigma_{\nu}(D)}.$$
We say a valuation $\nu_0\in \mathrm{Val}^*_X$ \emph{computes} $\lct_{\sigma}(X,\Delta,D)$ if
$$ \lct_{\sigma}(X,\Delta,D)=\frac{A_{X,\Delta}(\nu_{0})}{\sigma_{\nu_0}(D)}.$$
\end{enumerate}
\end{definition}

\medskip
\noindent We can also formulate similar conjectures for potential triples $(X,\Delta,D)$. 
The following is the main result of \cite{Xu2}, which gives an answer to the Weak conjecture.

\begin{theorem}[{cf. \cite[Theorem 1.1]{Xu2}}] \label{thm:valuation quasi-monomial}
Let $(X,\Delta)$ be a klt pair and $\mathfrak{a}_{\bullet}$ a graded sequence of ideals of $\mathcal{O}_X$ such that $\lct(X,\Delta,\mathfrak{a}_\bullet)<\infty$. Then there exists a quasi-monomial valuation $\omega$ which computes $\lct(X,\Delta,\mathfrak{a}_\bullet)$.
\end{theorem}

\begin{remark}\label{rmk:xu main result}
The strategy and the main ingredients of the proof of Theorem \ref{thm:valuation quasi-monomial} are as follows.
If $\nu$ is a valuation on $X$ which computes $\lct(X,\Delta,\mathfrak{a}_{\bullet})$, then there exists a quasi-monomial valuation $\omega$ which computes $\lct(X,\Delta,\mathfrak{a}_{\bullet})$ and satisfies the following properties:
\begin{enumerate}[leftmargin=8mm]
    \item[(1)] $A_{X,\Delta}(\nu)=A_{X,\Delta}(\omega)$,
    \item[(2)] $\omega\ge \nu$, and
    \item[(3)] $c_X(\nu)=c_X(\omega)$.
\end{enumerate}

\medskip

Our proof of Theorem \ref{thrm:1} is based on the generalization of these properties in the pklt setting.

The property (1) and (2) follow from \cite[Lemma 3.2]{Xu2}, while (3) is established in \cite[Proof of Theorem 1.1]{Xu2}. The property (4) follows from \cite[Proposition 3.1]{Xu2}.
\end{remark}

\begin{lemma} \label{lem:limit}
Let $(X,\Delta,D)$ be a potential triple such that $(X,\Delta)$ is lc and $A$ an ample divisor on $X$. Then
$$ \lct_{\sigma}(X,\Delta,D)=\lim_{\ell\to \infty}\lct_{\sigma}\left(X,\Delta,D+\frac{1}{\ell}A\right).$$
\end{lemma}

\begin{proof}
The lemma can be proved by observing the following sequence of equalities:
\begin{align*}
\lct_{\sigma}(X,\Delta,D)=\inf_{\nu\in \mathrm{Val}^*_X}\frac{A_{X,\Delta}(\nu)}{\sigma_{\nu}(D)}
 &=\inf_{\nu\in \mathrm{Val}^*_X}\lim_{\ell\to\infty}\frac{A_{X,\Delta}(\nu)}{\sigma_{\nu}\left(D+\frac{1}{\ell}A\right)}
\\ &=\lim_{\ell\to\infty}\inf_{\nu\in \mathrm{Val}^*_X}\frac{A_{X,\Delta}(\nu)}{\sigma_{\nu}\left(D+\frac{1}{\ell}A\right)}
\\ &=\lim_{\ell\to\infty}\lct_{\sigma}\left(X,\Delta,D+\frac{1}{\ell}A\right).
\end{align*}
Note that $\lim$ and $\inf$ are interchangeable since $\frac{A_{X,\Delta}(\nu)}{\sigma_{\nu}(D+\frac{1}{\ell}A)}$ monotonically decreases as $\ell$ increases.
\end{proof}

\section{Asymptotic multiplier ideal sheaves}\label{sect:asympt mult ideal}

In this section, we collect preliminary notions and results on asymptotic multiplier ideal sheaves. 

\subsection{Asymptotic multiplier ideal and log canonical threshold}\hfill
\smallskip

Let us first recall the definition of \textit{asymptotic multiplier ideal sheaf} \cite{Laz}.

\begin{definition}
Let $(X,\Delta)$ be a pair and $\mathfrak{a}$ an ideal of $\mathcal{O}_X$. Let $f\colon Y\to X$ be a log resolution of $(X,\Delta,\mf{a})$. We define the \emph{multiplier ideal sheaf} $\mathcal{J}(X,\Delta,\mathfrak{a}^{\lambda})$ of $(X,\Delta,\mathfrak{a}^{\lambda})$ for any $\lambda>0$ as
$$\mathcal{J}(X,\Delta,\mathfrak{a}^{\lambda})\coloneqq f_*\mathcal{O}_Y(\ceil{K_Y-f^*(K_X+\Delta)-\lambda E}).
$$
If $\mathfrak{a}_{\bullet}$ is a graded sequence of ideals of $\mathcal{O}_X$, then the \textit{asymptotic multiplier ideal sheaf} $\mathcal{J}(X,\Delta,\mathfrak{a}^{\lambda}_{\bullet})$ is defined as the maximal element of $\{\mathcal{J}(X,\Delta,\mathfrak{a}^{\frac{\lambda}{rm}}_{rm})\}_{m\in \Z_{>0}}$, where $r$ is a positive integer such that $r(K_X+\Delta)$ is Cartier. Such a maximal element exists by the standard argument (cf. \cite[Chapter 11]{Laz}).
\end{definition}

Let us prove the following lemmas.

\begin{lemma}[{cf. \cite[Lemma 1.7, Lemma 6.7]{JM}}] \label{lem:lct as min}
Let $(X,\Delta)$ be a klt pair and $\mathfrak{a},\mathfrak{q}$ ideals in $\mathcal{O}_X$. Let $f\colon Y\to X$ be a log resolution of the pair $(X,\Delta,\mathfrak{a}\cdot \mathfrak{q})$. If $E_1,\cdots,E_r$ are distinct prime divisors on $Y$ such that $ E_1\cup \cdots \cup E_r=\Supp f^{-1}_*\Delta\cup \Supp \mathcal{O}_Y/(\mathfrak{a}\cdot \mathfrak{q}),$
then
$$ \lct^{\mathfrak{q}}(X,\Delta,\mathfrak{a})=\min_{i=1,\cdots,r}\frac{A_{X,\Delta}(E_i)+\mathrm{ord}_{E_i}(\mathfrak{q})}{\mathrm{ord}_{E_i}(\mathfrak{a})}.$$
In particular, there exists a divisorial valuation $\nu=\ord_E$ which computes $ \lct^{\mathfrak{q}}(X,\Delta,\mathfrak{a})$.
\end{lemma}

\begin{proof}
Let us consider a function $\chi\colon \Val^*_X\to \R$ such that
$$ \chi(\nu)\coloneqq \frac{A_{X,\Delta}(\nu)+\nu(\mathfrak{q})}{\nu(\mathfrak{a})}$$
for any $\nu\in \mathrm{Val}^{*}_X$. Let $E\coloneqq \sum\limits^{r}_{i=1} E_i$ and $r_{(Y,E)}\colon \mathrm{Val}_X\to \mathrm{QM}(Y,E)$ the retraction map. Then we obtain the following:
$$
\begin{aligned} \left(\chi\circ r_{(Y,E)}\right)(\nu)&=\frac{A_{X,\Delta}(r_{(Y,E)}(\nu))+r_{(Y,E)}(\nu)(\mathfrak{q})}{r_{(Y,E)}(\nu)(\mathfrak{a})}\\
&=\frac{A_{X,\Delta}(r_{(Y,E)}(\nu))+\nu(\mathfrak{q})}{\nu(\mathfrak{a})}\\
&\le \frac{A_{X,\Delta}(\nu)+\nu(\mathfrak{q})}{\nu(\mathfrak{a})}\\
&=\chi(\nu),
\end{aligned}
$$
where the second equality follows from \cite[Corollary 4.8]{JM} and the following inequality holds from \cite[Corollary 5.4]{JM}.

This implies 
$$ \lct^{\mathfrak{q}}(X,\Delta,\mathfrak{a})=\inf_{\nu\in \mathrm{Val}^{*}_X}\chi(\nu)\ge \inf_{\nu\in \mathrm{QM}(Y,E)\setminus \{0\}}\chi(\nu)\ge \inf_{\nu\in \mathrm{Val}^{*}_X}\chi(\nu)=\lct^{\mathfrak{q}}(X,\Delta,\mathfrak{a}).$$
Thus, after rescaling $\nu$, we see that the following equalities hold:
$$ \lct^{\mathfrak{q}}(X,\Delta,\mathfrak{a})=\inf_{\nu\in \mathrm{QM}(Y,E)\setminus \{0\}}\chi(\nu)=\inf_{\nu \in \mathcal{D}(E)}\chi(\nu),$$
where the dual complex $\mathcal{D}(E)$ of $(Y,E)$ is identified with the subset of quasi-monomial valuations $\nu\in\mathrm{QM}(Y,E)$ such that $A_{X,\Delta}(\nu)=1$ as in Section \ref{dual complex}. Since $\mathcal{D}(E)$ is compact and $\chi\colon \mathcal{D}(E)\to \R$ is continuous, there exists a nonzero quasi-monomial valuation $\nu_0\in \mathcal{D}(E)$ computing $\lct^{\mathfrak{q}}(X,\Delta,\mathfrak{a})$. Suppose that $\nu_0=\nu_{\alpha,\eta}$ is a quasi-monomial valuation associated to a generic point $\eta=c_Y(\nu_0)$ of an irreducible component of $E_{i_1}\cap \cdots \cap E_{i_k}$ and $\alpha=(\alpha_{i_1},\cdots,\alpha_{i_k})\in \R^k_{\ge 0}$.

Let us consider a function $\phi\colon \QM_{\eta}(Y,E)\to \R$ such that for $\nu\in \mathrm{QM}_{\eta}(Y,E)$
$$ \phi(\nu)=A_{X,\Delta}(\nu)+\nu(\mathfrak{q})-\lct^{\mathfrak{q}}(X,\Delta,\mathfrak{a})\nu(\mathfrak{a}).$$
By \cite[Lemma 3.3]{Kim}, $\phi$ is a linear function on $\mathrm{QM}_{\eta}(Y,E)$. Similarly, we can also prove that $\nu\mapsto \nu(\mathfrak{q})$ is linear on $\mathrm{QM}_{\eta}(Y,E)$. Thus, $\phi\colon \mathrm{QM}_{\eta}(Y,E)\to \R$ is a linear function.

Note that $\phi(\nu)\ge 0$ for any $\nu\in \mathrm{QM}_{\eta}(Y,E)$, and $\phi(\nu)=0$ if and only if $\nu=0$ or $\nu$ computes $\lct^{\mathfrak{q}}(X,\Delta,\mathfrak{a})$. Since $\phi(\nu_0)=0$, any divisorial valuation $\mathrm{ord}_{E_{i_j}}$ (associated to $E_{i_j}$) with $\alpha_{i_j}\ne 0$ computes $\lct^{\mathfrak{q}}(X,\Delta,\mathfrak{a})$.
\end{proof}

\begin{lemma}[{cf. \cite[Section 1.4]{JM}}] \label{lem:lct as inf}
Let $(X,\Delta)$ be a klt pair and $\mathfrak{a},\mathfrak{q}$ ideals in $\mathcal{O}_X$. Then
$$ \lct^{\mathfrak{q}}(X,\Delta,\mathfrak{a})=\inf\{\lambda\mid\mathfrak{q}\nsubseteq \mathcal{J}(X,\Delta,\mathfrak{a}^{\lambda})\}.$$
\end{lemma}

\begin{proof}
Let $f\colon Y\to X$ be a log resolution of $(X,\Delta,\mathfrak{a}\cdot\mathfrak{q})$. Let $E_1,\cdots,E_r$ be distinct prime divisors on $Y$ such that
$$ E_1\cup \cdots \cup E_r=\Supp f^{-1}_*\Delta\cup \Supp \mathcal{O}_Y/(\mathfrak{a}\cdot \mathfrak{q}),$$
and $E$ a divisor on $Y$ such that $\mathfrak{a}\cdot \mathcal{O}_Y=\mathcal{O}_Y(-E)$. Note that the multiplier ideal sheaf $\mathcal{J}(X,\Delta,\mathfrak{a}^{\lambda})$ is the ideal sheaf in $\mathcal{O}_X$ consisting of those local sections $f\in \mathcal{O}_X$ such that
\begin{equation*}
-\mathrm{ord}_{E_i}(f)<\mathrm{ord}_{E_i}(\ceil{K_Y-f^*(K_X+\Delta)-\lambda E})
\end{equation*}
for all $i$, or equivalently,
$$ -\mathrm{ord}_{E_i}(f)\le A_{X,\Delta}(E_i)-\lambda \cdot \mathrm{ord}_{E_i}(\mathfrak{a})$$
for all $i$.

We first prove
$$
\lct^{\mathfrak{q}}(X,\Delta,\mathfrak{a})\ge \inf\{\lambda\mid \mathfrak{q}\nsubseteq \mathcal{J}(X,\Delta,\mathfrak{a}^{\lambda})\}.
$$
By Lemma \ref{lem:lct as min}, we have
$$ \lct^{\mathfrak{q}}(X,\Delta,\mathfrak{a})=\min_{i=1,\cdots,r}\frac{A_{X,\Delta}(E_i)+\mathrm{ord}_{E_i}(\mathfrak{q})}{\mathrm{ord}_{E_i}(\mathfrak{a})}.$$
Let $\lambda> \lct^{\mathfrak{q}}(X,\Delta,\mathfrak{a})$. Then there exists some $1\le i\le r$ such that $\lambda> \frac{A_{X,\Delta}(E_i)+\mathrm{ord}_{E_i}(\mathfrak{q})}{\mathrm{ord}_{E_i}(\mathfrak{a})}$. Therefore, for such an $i$ 
$$-\mathrm{ord}_{E_i}(\mathfrak{q})> A_{X,\Delta}(E_i)-\lambda \mathrm{ord}_{E_i}(\mathfrak{a}).$$
This implies that there is a local section $f\in \mathfrak{q}$ such that $f\notin \mathcal{J}(X,\Delta,\mathfrak{a}^{\lambda})$. Hence, $\mathfrak{q}\nsubseteq \mathcal{J}(X,\Delta,\mathfrak{a}^{\lambda})$, and the inequality holds.

For the converse, suppose that $\lambda<\lct^{\mathfrak{q}}(X,\Delta,\mathfrak{a})$. Then we have $\lambda<\frac{A_{X,\Delta}(E_i)+\mathrm{ord}_{E_i}(\mathfrak{q})}{\mathrm{ord}_{E_i}(\mathfrak{a})}$ for any $1\le i\le r$, and thus
$$ -\mathrm{ord}_{E_i}(\mathfrak{q})<A_{X,\Delta}(E_i)-\lambda \mathrm{ord}_{E_i}(\mathfrak{a})\text{ for all }1\le i\le r.$$
Therefore, we obtain $\mathfrak{q}\subseteq \mathcal{J}(X,\Delta,\mathfrak{a}^{\lambda})$, which implies
\begin{align*}
\lct^{\mathfrak{q}}(X,\Delta,\mathfrak{a})\le \inf\{\lambda\mid \mathfrak{q}\nsubseteq \mathcal{J}(X,\Delta,\mathfrak{a}^{\lambda})\}.
\end{align*}
This completes the proof.
\end{proof}

For a pair $(X,\Delta)$, a graded sequence of ideals $\mathfrak{a}_{\bullet}$ and a positive integer $m$, we denote
$$\mathfrak{b}_m\coloneqq\mathcal{J}(X,\Delta,\mathfrak{a}^m_{\bullet}).$$
For a graded sequence of ideals $\mathfrak{a}_{\bullet}$, we let
$\Phi(\mathfrak{a}_{\bullet})\coloneqq\{m\in \Z_{>0}\mid \mathfrak{a}_m\ne (0)\}$.

\begin{corollary}[{\cite[Lemma 1.58]{Xu4}}]\label{cor:tak}
Let $(X,\Delta)$ be a klt pair and $\mathfrak{a}_{\bullet}$ a graded sequence of ideals of $\mathcal{O}_X$. There is a nonzero ideal $I\subseteq \mathcal{O}_X$ which depends only on $(X,\Delta)$ such that for any $m,m'\in \Phi(\mathfrak{a}_{\bullet})$, we have $I\cdot \mathfrak{b}_{m+m'}\subseteq \mathfrak{b}_m\cdot \mathfrak{b}_{m'}$.
\end{corollary}

\begin{lemma}[{cf. \cite[Lemma 1.59 (ii)]{Xu4}}] \label{lem:asymptotic2}
Let $(X,\Delta)$ be a klt pair, $\mathfrak{q}\subseteq \mathcal O_X$ an ideal, and $\mathfrak{a}_{\bullet}$ a graded sequence of ideals on $X$. Then we have
\begin{equation}\tag{$\star$}\label{star}
\lim_{m\to \infty}\{m\cdot \lct^{\mathfrak{q}}\left(X,\Delta,\mathfrak{b}_m\right)\}=\sup_{m\ge 1} \{m\cdot \lct^{\mathfrak{q}}(X,\Delta,\mathfrak{a}_{m})\}.
\end{equation}
\end{lemma}

\begin{proof}
Since $\lct^{\mathfrak{q}}(X,\Delta,\mathfrak{b}^{\frac{1}{m}}_m)\ge \lct^{\mathfrak{q}}(X,\Delta,\mathfrak{a}^{\frac{1}{m}}_m)$ for any positive integer $m$, the inequality $\geq$ holds in (\ref{star}). We may also assume that  $\sup\limits_{m\ge 1} m\cdot \lct^{\mathfrak{q}}(X,\Delta,\mathfrak{a}_{m})<\infty$.  Let $m'$ be a  sufficiently divisible positive integer. By Lemma \ref{lem:lct as inf}, there exists a divisorial valuation $\nu=\ord_E$ that computes $\lct^{\mathfrak{q}}(X,\Delta,\mathfrak{a}_{m'})$ such that $\lct^{\mathfrak{q}}(X,\Delta,\mathfrak{a}_{m'})=\frac{A_{X,\Delta}(E)+\ord_E(\mathfrak{q})}{\ord_E(\mathfrak{a}_{m'})}$. Thus, we have
\begin{align*}
    \lct^{\mathfrak{q}}(X,\Delta,\mathfrak{b}^{\frac{1}{m}}_m)&\le \frac{A_{X,\Delta}(E)+\ord_E(\mathfrak{q})}{\frac{1}{m}\ord_E(\mathfrak{b}_m)}
    \\ &\le \frac{A_{X,\Delta}(E)+\ord_E(\mathfrak{q})}{\frac{1}{m'}\ord_E(\mathfrak{a}_{m'})-\frac{1}{m}(A_{X,\Delta}(E)+\ord_E(\mathfrak{q}))}
    \\ &=\frac{m\cdot \lct^{\mathfrak{q}}(X,\Delta,\mathfrak{a}^{\frac{1}{m'}}_{m'})}{m-\lct^{\mathfrak{q}}(X,\Delta,\mathfrak{a}^{\frac{1}{m'}}_{m'})}.
\end{align*}
By letting $m\to \infty$, we obtain the inequality $\leq$ in (\ref{star}).
\end{proof}

\begin{proposition}\label{prop:sup}
Let $(X,\Delta)$ be a klt pair, $\mathfrak{q}\subseteq \mathcal{O}_X$ an ideal, and $\mathfrak{a}_{\bullet}$ a graded sequence of ideals of $\mathcal{O}_X$. Then we have
$$ \lct^{\mathfrak{q}}(X,\Delta,\mathfrak{a}_{\bullet})=\sup_{m\ge 1}\{m\cdot \lct^{\mathfrak{q}}(X,\Delta,\mathfrak{a}_m)\}.$$
\end{proposition}

\begin{proof}
We argue in a similar way as in the proof of \cite[Lemma 1.60]{Xu4}. For simplicity, we denote by 
$\lct^{\mathfrak{q}}_s(X,\Delta,\mathfrak{a}_{\bullet})\coloneqq\sup\limits_{m\ge 1}\{m\cdot \lct^{\mathfrak{q}}(X,\Delta,\mathfrak{a}_m)\}.$

Since we have $\nu(\mathfrak{a}_{\bullet})\le \frac{\nu(\mathfrak{a}_m)}{m}$ for any $\nu\in \mathrm{Val}_X$, we obtain the inequality $\lct^{\mathfrak{q}}(X,\Delta,\mathfrak{a}_{\bullet})\geq m\cdot\lct^{\mathfrak{q}}(X,\Delta, \mathfrak{a}_m).$
Thus, we obtain
$$\lct^{\mathfrak{q}}(X,\Delta,\mathfrak{a}_{\bullet})\ge \lct^{\mathfrak{q}}_s(X,\Delta,\mathfrak{a}_{\bullet}).$$

On the other hand, we may assume that $\lct^{\mathfrak{q}}_s(X,\Delta,\mathfrak{a}_m)<\infty$. By Corollary \ref{cor:tak}, for any positive integer $n$, there is an ideal $I\subseteq \mathcal{O}_X$ such that
$$ I^n\cdot \mathfrak{b}_{(n+1)!}\subseteq (\mathfrak{b}_{n!})^{n+1}.$$
Therefore, by \cite[Lemma 1.59 (i)]{Xu4}, we have the following inequalities 
\begin{equation*} 
\begin{aligned}
\frac{1}{n!}\ord_E(\mathfrak{b}_{n!})&\le \frac{n}{(n+1)!}\ord_E(I)+\frac{1}{(n+1)!}\ord_E(\mathfrak{b}_{(n+1)!})
\\ & \le \sum_{n'\ge n}\frac{n'}{(n'+1)!}\ord_E(I)+\ord_E(\mathfrak{a}_{\bullet})
\\ & = \frac{1}{n!}\ord_E(I)+\ord_E(\mathfrak{a}_{\bullet}).
\end{aligned}
\end{equation*}
Fix a constant $C$ such that $\ord_E(I)\le C\cdot A_{X,\Delta}(E)$ for any prime divisor $E$ over $X$. 
For example, we can take $C=1/\lct(I)$ where $\lct(I)=\inf_E\frac{A_{X,\Delta}(E)}{\ord_E(I)}$. Note that it is positive since $(X,\Delta)$ is klt. Then, by Lemma \ref{lem:asymptotic2}, for any $\varepsilon>0$, there exists a positive integer $n$ such that $\frac{C}{n!}<\frac{\varepsilon}{2}$ and
$$ \left|\frac{1}{n!\cdot \lct^{\mathfrak{q}}(X,\Delta,\mathfrak{b}_{n!})}-\frac{1}{\lct^{\mathfrak{q}}_s(X,\Delta,\mathfrak{a}_{\bullet})}\right|\le \frac{\varepsilon}{2}.$$
For any divisor $E$ computing $\lct^{\mathfrak{q}}(X,\Delta,\mathfrak{b}_{n!})$, by the above inequality, we have
$$ \frac{1}{n! \cdot \lct^{\mathfrak{q}}(X,\Delta,\mathfrak{b}_{n!})}=\frac{\ord_E(\mathfrak{b}_{n!})}{n!\cdot A_{X,\Delta}(E)+n!\ord_E(\mathfrak{q})}\le \frac{\ord_E(\mathfrak{a}_{\bullet})}{A_{X,\Delta}(E)+\ord_E(\mathfrak{q})}+\frac{\varepsilon}{2}.$$
Thus, we obtain the following:

\begin{small}
\begin{align*}
0&\le \frac{1}{\lct^{\mathfrak{q}}_s(X,\Delta,\mathfrak{a}_{\bullet})}-\frac{\ord_E(\mathfrak{a}_{\bullet})}{A_{X,\Delta}(E)+\ord_E(\mathfrak{q})}&
\\ &=\left(\frac{1}{\lct^{\mathfrak{q}}_s(X,\Delta,\mathfrak{a}_{\bullet})}-\frac{1}{n!\cdot \lct^{\mathfrak{q}}(X,\Delta,\mathfrak{b}_{n!})}\right)+\left(\frac{1}{n!\cdot \lct^{\mathfrak{q}}(X,\Delta,\mathfrak{b}_{n!})}-\frac{\ord_E(\mathfrak{a}_{\bullet})}{A_{X,\Delta}(E)+\ord_E(\mathfrak{q})}\right)&
\\ &\le \varepsilon.
\end{align*}
\end{small}
This implies $\lct^{\mathfrak{q}}(X,\Delta,\mathfrak{a}_{\bullet})=\sup_{m\ge 1}\{m\cdot \lct^{\mathfrak{q}}(X,\Delta,\mathfrak{a}_m)\}.$
\end{proof}

\begin{theorem}[{cf. \cite[Proposition 2.12, Corollary 6.9]{JM}}] \label{thrm:sup and asymptotic}
Let $(X,\Delta)$ be a pair, $\mathfrak{q}\subseteq \mathcal{O}_X$ an ideal, and $\mathfrak{a}_{\bullet}$ a graded sequence of ideals of $\mathcal{O}_X$. Then
$$ \lct^{\mathfrak{q}}(X,\Delta,\mathfrak{a}_{\bullet})=\inf\{\lambda\mid \mathfrak{q}\nsubseteq \mathcal{J}(X,\Delta,\mathfrak{a}^{\lambda}_{\bullet})\}.$$
\end{theorem}

\begin{proof}
Let us first prove the following inequality
\begin{equation*}
\lct^{\mathfrak{q}}(X,\Delta,\mathfrak{a}_{\bullet})\le \inf \{\lambda\mid \mathfrak{q}\nsubseteq \mathcal{J}(X,\Delta,\mathfrak{a}^{\lambda}_{\bullet})\}.
\end{equation*} 
By Proposition \ref{prop:sup}, we have 
$\lct^{\mathfrak{q}}(X,\Delta,\mathfrak{a}_{\bullet})=\sup\limits_m \{m\cdot \lct^{\mathfrak{q}}(X,\Delta,\mathfrak{a}_m)\}.$
Let $\lambda$ be a positive number such that $\lct^{\mathfrak{q}}(X,\Delta,\mathfrak{a}_{\bullet})>\lambda$. Then we have $\lct^{\mathfrak{q}}(X,\Delta,\mathfrak{a}_{m_0})>\frac{\lambda}{m_0}$ and Lemma \ref{lem:lct as inf} gives us 
$$\mathfrak{q}\subseteq \mathcal{J}(X,\Delta,\mathfrak{a}^{\lambda'}_{\bullet})=\mathcal{J}\left(X,\Delta,\mathfrak{a}^{\frac{\lambda'}{m_0}}_{m_0}\right)$$
for any $\lambda'\le \lambda$ and any sufficiently large $m_0$. This proves the above inequality.

For the converse, note that the following inequality holds:
$$ \lct^{\mathfrak{q}}(X,\Delta,\mathfrak{a}_{m})\le \frac{\lct^{\mathfrak{q}}(X,\Delta,\mathfrak{a}_{\bullet})}{m}$$
for all positive integers $m$. Therefore, Lemma \ref{lem:lct as inf} implies $\mathfrak{q}\nsubseteq \mathcal{J}\left(X,\Delta,\mathfrak{a}^{\frac{\lct^{\mathfrak{q}}(X,\Delta,\mathfrak{a}_{\bullet})}{m}}_m\right).$
Thus, by the definition of asymptotic multiplier ideal, we have $\mathfrak{q}\nsubseteq \mathcal{J}\left(X,\Delta,\mathfrak{a}^{\lct^{\mathfrak{q}}(X,\Delta,\mathfrak{a}_{\bullet})}_{\bullet}\right)$. Hence, the following reverse inequality is also proved
\[ \lct^{\mathfrak{q}}(X,\Delta,\mathfrak{a}_{\bullet})\ge \inf \{\lambda\mid \mathfrak{q}\nsubseteq \mathcal{J}(X,\Delta,\mathfrak{a}^{\lambda}_{\bullet})\}. \qedhere\]
\end{proof}

\medskip

\subsection{Asymptotic multiplier ideal for $\Q$-divisors}
Let us recall the definition of asymptotic multiplier ideals for a $\Q$-divisor $D$. For details, see \cite{Laz}, \cite{Leh}, \cite{CJK}.

\begin{definition}
Let $(X,\Delta)$ be a projective klt pair and $D$ a pseudoeffective $\Q$-divisor on $X$. Consider the following graded sequence of ideals 
$$ \mathfrak{a}_{m}\coloneqq\begin{cases}\mathfrak{b}(|mD|), & \text{if }mD\text{ is Cartier}, \\ (0),& \text{otherwise}.\end{cases}$$
We define
$\mathcal{J}(X,\Delta,\lambda\|D\|)\coloneqq\mathcal{J}(X,\Delta,\mathfrak{a}^{\lambda}_{\bullet})$
for any positive number $\lambda>0$. If $X$ is smooth and $\Delta=0$, then we denote $\mathcal{J}(\|D\|)\coloneqq \mathcal{J}(X,0,\|D\|)$.
\end{definition}

Let us prove the following Nadel type vanishing theorem, which essentially follows from \cite[Theorem 11.2.12]{Laz}.

\begin{proposition} \label{prop:Nadel}
Let $(X,\Delta)$ be a projective klt pair and $D$ an effective $\Q$-divisor on $X$. Let $L$ be a Cartier divisor such that $L=K_X+\Delta+\lambda D+A$ for an ample $\R$-divisor $A$ on $X$. Fix a positive number $\lambda>0$. Then we have
$$ H^i(X,\mathcal{O}_X(L)\otimes \mathcal{J}(X,\Delta,\lambda\|D\|))=0\text{ for any }i>0.$$
\end{proposition}

\begin{proof}
Fix a positive integer $m$ such that $mD$ is Cartier and
$$ \mathcal{J}(X,\Delta,\lambda\|D\|)=\mathcal{J}\left(X,\Delta,\frac{\lambda}{m}|mD|\right).$$
Let $f\colon Y\to X$ be a log resolution of $(X,\Delta)$. Let $f^*(mD)=M_m+F_m$ be a decomposition of $f^*(mD)$ into a movable part $M_m$ and the fixed part $F_m$. We may further assume that $M_m$ is semiample and $F_m$ has a simple normal crossing support. By definition, we have
\begin{equation} \label{Nadel-1}
\mathcal{J}(X,\Delta,\lambda\|D\|)=f_*\mathcal{O}_Y\left(K_Y-\floor{f^*(K_X+\Delta)+\frac{\lambda}{m} F_m}\right).
\end{equation}
Note that we have the following equalities 
\begin{align*}
&\mathcal{O}_Y(f^*L)\otimes \mathcal{O}_Y\left(K_Y-\floor{f^*(K_X+\Delta)+\frac{\lambda}{m}F_m}\right)\\
&\phantom{.................................}=\mathcal{O}_Y\left(K_Y-\floor{f^*(K_X+\Delta-L)+\frac{\lambda}{m}F_m}\right)\\
&\phantom{.................................}=\mathcal{O}_Y\left(K_Y-\floor{-\lambda f^*D-f^*A+\frac{\lambda}{m}F_m}\right)\\
&\phantom{.................................}=\mathcal{O}_Y\left(K_Y+\ceil{\frac{\lambda}{m}M_m+f^*A}\right).
\end{align*}
Moreover, by the relative Kawamata--Viehweg vanishing theorem (cf. \cite[Theorem 3.2.9]{Fuj}), we obtain the following vanishing of cohomologies for any $i>0$
\begin{equation} \label{Nadel-2} H^i\left(Y,\mathcal{O}_Y(f^*L)\otimes \mathcal{O}_Y\left(K_Y-\floor{f^*(K_X+\Delta)+\frac{\lambda}{m}F_m}\right)\right)=0
\end{equation}
and
\begin{equation} \label{Nadel-3} R^if_*\left(\mathcal{O}_Y(f^*L)\otimes \mathcal{O}_Y\left(K_Y-\floor{f^*(K_X+\Delta)+\frac{\lambda}{m}F_m}\right)\right)=0.
\end{equation}
Now we apply the Leray spectral sequence
$$
\begin{aligned}
E^{st}_2&=H^s\left(X,R^tf_*\left(\mathcal{O}_Y(f^*L)\otimes \mathcal{O}_Y\left(K_Y-\floor{f^*(K_X+\Delta)+\frac{\lambda}{m}F_m}\right)\right)\right)
\\ &\Longrightarrow H^{s+t}\left(Y,\mathcal{O}_Y(f^*L)\otimes \mathcal{O}_Y\left(K_Y-\floor{f^*(K_X+\Delta)+\frac{\lambda}{m}F_m}\right)\right).
\end{aligned}$$
Combining (\ref{Nadel-1}) with (\ref{Nadel-2}) and (\ref{Nadel-3}), we obtain the assertion.
\end{proof}

Below, we list some consequences that follow easily from the definitions and the results obtained above.

\begin{proposition}\label{prop:easy properties}
Let $(X,\Delta)$ be a klt pair, $D$ a big $\Q$-divisor on $X$ and $A,A'$ ample divisors on $X$. Fix positive numbers $c'>c>0$. Then we have the following properties.
\begin{enumerate}[leftmargin=8mm]
    \item[$\mathrm{(1)}$]
    The inclusion holds:
    $$\mathcal{J}(X,\Delta,c'\|D\|)\subseteq \mathcal{J}(X,\Delta,c\|D\|).$$
    \item[$\mathrm{(2)}$] 
    The inclusion holds for any sufficiently large integer $\ell>0$:
    $$
    \mathcal{J}\left(X,\Delta,c\left\|D+\frac{1}{\ell!}A'\right\|\right)\subseteq \mathcal{J}\left(X,\Delta,c\left\|D+\frac{1}{(\ell-1)!}A\right\|\right).
    $$
    \item[$\mathrm{(3)}$] The sequence 
$$ \left\{\mathcal{J}\left(X,\Delta,c\left\|D+\frac{1}{\ell !}A\right\|\right)\right\}_{\ell\in \Z_{>0}}$$
forms a decreasing sequence of ideals which stabilizes for all sufficiently large $\ell>0$.
\end{enumerate}
\end{proposition}

\begin{proof}
(1) Immediate from definition.

\noindent (2) For any sufficiently large integer $\ell> 0$, $\ell A-A'$ is very ample. Thus,
    $$ \mathfrak{b}(|m\ell! D+mA'|)\subseteq \mathfrak{b}(|m\ell!D+m\ell A|)$$
    for any positive integer $m$.

\noindent (3)  For any sufficiently large integer $m>0$ and any two positive integers $m',\ell$, we have 
$$\mathfrak{b}\left(\left|m\ell m'D+mA\right|\right)\subseteq \mathfrak{b}(|m\ell m'D+mm'A|).$$
Thus, we have the inclusion
$$ \mathcal{J}\left(X,\Delta,c\left\|D+\frac{1}{\ell m'}A\right\|\right)\subseteq \mathcal{J}\left(X,\Delta,c\left\|D+\frac{1}{\ell}A\right\|\right).$$
The sequence stabilizes by Lemma \ref{lem:diminished stable}.
\end{proof}

\begin{lemma}\label{lem:diminished stable}
Let $(X,\Delta)$ be a projective klt pair, $D$ a pseudoeffective $\Q$-divisor on $X$, and $A$ an ample divisor on $X$. Fix a real number $c\ge 1$. Then there exists a positive integer $\ell_0$ such that
$$ \mathcal{J}\left(X,\Delta,c\left\|D+\frac{1}{\ell!}A\right\|\right)=\mathcal{J}\left(X,\Delta,c\left\|D+\frac{1}{\ell_0!}A\right\|\right)$$
for any integer $\ell\geq \ell_0$.
\end{lemma}

\begin{proof}
The following proof is inspired by the proofs of \cite[Proposition 5.1]{Hac} and \cite[Theorem 4.2]{Leh}.

Let $A'$ be an ample Cartier divisor on $X$ such that
$$A'-\left(K_X+\Delta+c\left(D+\frac{1}{\ell!}A\right)\right)$$ is ample for any positive integer $\ell$. Let $H$ be a very ample divisor on $X$. By Proposition \ref{prop:Nadel}, we have
$$ H^i\left(X,\mathcal{J}\left(X,\Delta,c\left\|D+\frac{1}{\ell!}A\right\|\right)\otimes \mathcal{O}_X(A'+nH)\right)=0\text{ for any }i\ge 1 \text{ and any } n\in \Z_{>0}.$$
Then, by the Castelnuovo--Mumford regularity, the sheaf
$$ \mathcal{J}\left(X,\Delta,c\left\|D+\frac{1}{\ell!}A\right\|\right)\otimes \mathcal{O}_X(A'+nH)$$
is globally generated for $n\geq \dim X+1$. Any globally generated coherent sheaf $\mathcal{F}\subseteq \mathcal{O}_X(A'+(\dim X+1)H)$ corresponds to the subspace 
$$H^0(X,\mathcal{F})\subseteq H^0(X,A'+(\dim X+1)H).$$ Hence, there are at most finitely many distinct elements in
$$\left\{\mathcal{J}\left(X,\Delta,c\left\|D+\frac{1}{\ell !}A\right\|\right)\right\}_{\ell\in \Z_{>0}.}$$
Hence, this sequence of asymptotic multiplier ideal sheaves is eventually constant, and we obtain the result.
\end{proof}

The following allows us to consider asymptotic multiplier ideals for pseudoeffective divisor $D$ on a pair $(X,\Delta)$.

\begin{definition}[{\cite[Definition 6.2]{Leh}}] \label{definition:diminished multiplier ideal}
Let $(X,\Delta)$ be a projective klt pair, $D$ a pseudoeffective $\Q$-divisor and $A$ an ample divisor on $X$. Fix a positive number $c>0$.

\begin{enumerate}[(1), leftmargin=8mm]
    \item We denote by $\mathcal{J}_-(X,\Delta,c\|D\|)$ the minimal element of the set 
    $$ \left\{\mathcal{J}\left(X,\Delta,c\left\|D+\frac{1}{\ell !}A\right\|\right)\right\}_{\ell\in \Z_{>0}}.$$
    The minimal exists and it is independent of the choice of $A$ by Proposition \ref{prop:easy properties}(2). If $X$ is smooth and $\Delta=0$, then we denote by
    $$ \mathcal{J}_-(c\|D\|)\coloneqq \mathcal{J}(X,0,c\|D\|).$$
    \item For a fixed $c>0$, the set
    $\{\mathcal{J}_-(X,\Delta,c'\|D\|)\}_{c'>c}$
    has the maximal element which we denote by $\mathcal{J}_{\sigma}(X,\Delta,c\|D\|)$ and call the \emph{diminished multiplier ideal}. Note that the set admits the maximal element by Proposition \ref{prop:easy properties} (1).
\end{enumerate}
\end{definition}

\subsection{Continuity of asymptotic order} We prove the following continuity of a function, which will be used in the proof of Theorem \ref{thrm:1}.

\begin{proposition} \label{prop:continuity}
Let $(X,\Delta)$ be a projective klt pair and let $D$ be a pseudoeffective $\Q$-Cartier $\Q$-divisor on $X$. Let $M$ be a positive integer. We denote by
$$ \mathrm{Val}^{\le M}_X\coloneqq \{\nu\in \mathrm{Val}_X\mid A_{X,\Delta}(\nu)\le M\}.$$
Then the function
$$ \mathrm{Val}^{\le M}_{X}\to \R,\quad \nu\mapsto \sigma_{\nu}(D) $$
is continuous.
\end{proposition}

\begin{proof}
We first prove that we may assume that $X$ is smooth and $\Delta=0$. Let $f\colon Y\to X$ be a resolution of $(X,\Delta)$ and $\Delta_Y$ be a divisor on $Y$ such that $K_Y+\Delta_Y=f^*(K_X+\Delta)$. 
Note that since the reduced divisor $(\Delta_Y)_{\mathrm{red}}$ is simple normal crossing, $(Y,(\Delta_Y)_{\mathrm{red}})$ is lc. Moreover, since $(X,\Delta)$ is klt, there exists $\varepsilon>0$ such that $\Delta_Y\le (1-\varepsilon)(\Delta_Y)_{\mathrm{red}}$. Hence, by \cite[Lemma 2.9]{HLQ}, we obtain that
$$ 
\begin{aligned}
A_{X,\Delta}(\nu)&=A_{Y}(\nu)-\nu(\Delta_Y)\\ 
&\ge A_{Y}(\nu)-(1-\varepsilon) \nu((\Delta_Y)_{\mathrm{red}})\\ 
&=\varepsilon A_{Y}(\nu)+(1-\varepsilon)A_{Y,(\Delta_Y)_{\mathrm{red}}}(\nu)\\ 
&\ge \varepsilon A_{Y}(\nu).
\end{aligned}
$$
Therefore, by replacing $(X,\Delta)$ with $(Y,0)$ and $M$ with $\frac{M}{\varepsilon}$, we may assume that $X$ is smooth and $\Delta=0$. Note that we can identify the spaces $\Val_X=\Val_Y$ using the homeomorphism given by the birational morphism $f:Y\to X$.

Let $A$ be an ample Cartier divisor on $X$, and let $L$ be a big $\Q$-Cartier $\Q$-divisor on $X$. Then, by \cite[Propositions 2.13, 6.2, and 6.5]{JM}, for every $\nu\in \mathrm{Val}^{\le M}_X$, we have
\begin{equation}\label{ineq-conti}
0\le \sigma_{\nu}(L)-\frac{1}{m}\nu(\mathcal{J}(\|mL\|))\le \frac{A_X(\nu)}{m}.
\end{equation}
\noindent Indeed, for $t>0$, set $j_t(L)\coloneqq \mathcal{J}(X,0,\mathfrak{a}_{\bullet}(L)^t)$. For a sufficiently divisible $q$, we have $j_t(L)=\mathcal{J}(X,0,\mathfrak{a}_q(L)^{\frac{t}{q}})$. Taking $q=mp$ gives
\begin{equation}\label{eq:subadd}
    j_m(L)=\mathcal{J}(X,0,\mathfrak{a}_{mp}(L)^{\frac{1}{p}})=\mathcal{J}(X,0,\|mL\|)=\mathcal{J}(\|mL\|).
\end{equation}
By \cite[Proposition 2.13]{JM}, $j_{\bullet}(L)$ has controlled growth and since $A_X(\nu)<\infty$, \cite[Proposition 6.2]{JM} implies
\begin{equation}\label{eq:JM_subadd}
    \nu(j_{\bullet}(L))=\nu(\mathfrak{a}_{\bullet}(L))=\sigma_{\nu}(L),
\end{equation}
where $\nu(j_{\bullet}(L))\coloneqq \lim_{t\to\infty} \frac{\nu(j_t(L))}{t}$. By \cite[Proposition 6.5]{JM}, we obtain
$$
\nu(j_{\bullet}(L))-\frac{A_X(\nu)}{m}\leq \frac{1}{m}\nu(j_m(L))\leq \nu(j_{\bullet}(L)),
$$
which together with (\ref{eq:subadd}) and (\ref{eq:JM_subadd}) implies (\ref{ineq-conti}).

Note that we can choose a sequence of positive numbers $\delta_i$ such that $\lim_{i\to\infty}\delta_i= 0^+$ and 
$$ 
\mathcal{J}_{-}(\|mD\|)=\mathcal{J}(\|mD+\delta_i A\|)=\mathcal{J}\left(\left\|m\left(D+\frac{\delta_i}{m}A\right)\right\|\right).
$$
Thus, by (\ref{ineq-conti}), we have
\begin{equation} \label{ineq-conti2}
\sigma_{\nu}\left(D+\frac{\delta_i}{m}A\right)-\frac{A_X(\nu)}{m}\le \frac{1}{m}\nu(\mathcal{J}_{-}(mD))\le \sigma_{\nu}\left(D+\frac{\delta_i}{m}A\right).
\end{equation}
Taking the limit $i\to\infty$ in (\ref{ineq-conti2}) gives
$$ \sigma_{\nu}(D)-\frac{A_X(\nu)}{m}\le \frac{1}{m}\nu(\mathcal{J}_{-}(mD))\le \sigma_{\nu}(D).$$
Note that $\nu\mapsto \nu(\mathcal{J}_{-}(mD))$ is continuous on $\mathrm{Val}_X$ by the definition of the topology on $\mathrm{Val}_X$, and
$$ \left|\sigma_{\nu} (D)-\frac{1}{m}\nu(\mathcal{J}_{-}(mD))\right|\le \frac{A_X(\nu)}{m}$$
holds for every $\nu\in \mathrm{Val}^{\le M}_X$. Now, the proposition follows from the fact that a uniform limit of continuous functions is continuous.
\end{proof}

\subsection{Enlarged graded sequence of ideals}

We provide a version of \cite[Proposition 7.14]{JM} in Proposition \ref{prop:enlarged ideal for lct}. Lemma \ref{lem:base change of multiplier ideal} shows the local-to-global relations between the multiplier ideal sheaves and the log canonical thresholds.

We first prove the following lemma, which is a slight generalization of \cite[Proposition 1.9]{JM}.

\begin{lemma}\label{lem:base change of multiplier ideal}
For a pair $(X,\Delta)$ with a point $x$ in $X$, let $\varphi\colon X_x\to X$ be the natural morphism. Suppose that $\mathfrak{a}$ and $\mathfrak{q}$ are nonzero ideals in $\mathcal{O}_X$. The following properties hold for any positive number $\lambda'$:
\begin{enumerate}[leftmargin=8mm]
    \item[$\mathrm{(1)}$] $\mathcal{J}(X_x,\Delta_x,\mathfrak{a}^{\lambda'}_x)=\mathcal{J}(X,\Delta,\mathfrak{a}^{\lambda'})\cdot \mathcal{O}_{X,x},$ and 
    \item[$\mathrm{(2)}$] $\lct^{\mathfrak{q}}(X,\Delta,\mathfrak{a})\le \lct^{\mathfrak{q}_x}(X_x,\Delta_x,\mathfrak{a}_x)$, where the equality holds if $\mathcal{Z}((\mathcal{J}(X,\Delta,\mathfrak{a}^{\lambda}):\mathfrak{q}))\cap \varphi(X_x)\ne \varnothing$ with $\lambda=\lct^{\mathfrak{q}}(X,\Delta,\mathfrak{a})$.
\end{enumerate}
\end{lemma}

\begin{proof}
(1) Let $f\colon Y\to X$ be a log resolution of both $\mathfrak{a}$ and $(X,\Delta)$. There exists an effective divisor $E$ on $Y$ such that $\mathfrak{a}\cdot \mathcal{O}_Y=\mathcal{O}_Y(-E)$. For $Y'\coloneqq Y\times_X X_x$, if $g\colon Y'\to X_x$ and $h\colon Y'\to Y$ are natural projections, then $\mathfrak{a}_x\cdot \mathcal{O}_{Y'}=\mathcal{O}_{Y'}(-h^*E)$ and $$K_{Y'}-g^*(K_{Y'}+h^*\Delta)-\lambda'h^*E=h^*(K_Y-f^*(K_Y+\Delta)-\lambda'E)$$
is a simple normal crossing divisor. Therefore, the flat base change theorem (cf. \cite[Lemma 02KH]{Stacks}) gives us (1).

\noindent (2) The inclusion $\mathfrak{q}\subseteq \mathcal{J}(X,\Delta,\mathfrak{a}^{\lambda'})$ for any $0\le \lambda'<\lambda$ implies $\mathfrak{q}_x\subseteq \mathcal{J}(X,\Delta,\mathfrak{a}^{\lambda'})\cdot \mathcal{O}_{X,x}$. Therefore, (1) and Lemma \ref{lem:lct as inf} give us the inequality in (2).
Since $$(\mathcal{J}(X_x,\Delta_x,\mathfrak{a}^{\lambda}_x):\mathfrak{q}_x)=(\mathcal{J}(X,\Delta,\mathfrak{a}^{\lambda}):\mathfrak{q})\cdot \mathcal{O}_{X,x}\ne \mathcal{O}_{X,x}$$
holds, we can prove the reverse inequality by assuming $\mathfrak{q}_x\nsubseteq \mathcal{J}(X_x,\Delta_x,\mathfrak{a}^{\lambda}_x)$ and applying Lemma \ref{lem:lct as inf}.
\end{proof}
\medskip

The following Proposition \ref{prop:enlarged ideal for lct} will be used in the proof of Theorem \ref{thrm:1}.

\begin{proposition}[{cf. \cite[Proposition 7.14]{JM}}] \label{prop:enlarged ideal for lct}
Let $(X,\Delta,D)$ be a pklt triple and $x$ the generic point of an irreducible component of $\mathcal{Z}(\mathcal{J}_{\sigma}(X,\Delta,\lct_{\sigma}(X,\Delta,D)\|D\|))$. Then there is a positive number $\varepsilon_0>0$, a positive integer $p$ and a graded sequence of ideals $\mathfrak{c}_{\bullet,\ell}$ in $\mathcal{O}_X$ such that for any $0<\varepsilon'_0<\frac{\varepsilon_0}{2}$, there is a positive integer $\ell_0\coloneqq \ell_0(\varepsilon'_0)$ which satisfies the following properties: for each $\ell\ge \ell_0$ and for a fixed ample divisor $A$ on $X$,
\begin{enumerate}[leftmargin=8mm]
\item[$\mathrm{(1)}$] we have
$$ \lct_{\sigma}\left(X,\Delta,D+\frac{1}{\ell!}A\right)\le \lct(X_x,\Delta_x,(\mathfrak{c}_{\bullet,\ell!})_x)\le (1+\varepsilon'_0)\cdot\lct_{\sigma}\left(X,\Delta,D+\frac{1}{\ell!} A\right),$$

\item[$\mathrm{(2)}$] there exists a positive integer $p$ (which is independent of $\varepsilon'_0$ and $\ell$) such that $\mathfrak{m}^p_x\subseteq \mathfrak{c}_{1,\ell!}$, and

\item[$\mathrm{(3)}$]  $\nu(\mathfrak{c}_{\bullet,\ell!})=\min\left\{\sigma_{\nu}\left(D+\frac{1}{\ell!}A\right),p\cdot \nu(\mathfrak{m}_x)\right\}$ 
for all $\nu\in \mathrm{Val}^*_X$.
\end{enumerate}
\end{proposition}

\begin{proof}
We prove the statements in four steps.
\smallskip

\noindent\textbf{Step 1.} 
Fix an ample divisor $A$ on $X$. We show in this step  that there is a positive number $\varepsilon_0>0$ such that for any $0<\varepsilon'_0<\frac{\varepsilon_0}{2}$, there is a positive integer $\ell'_0\coloneqq \ell'_0(\varepsilon_0,\varepsilon'_0)$ satisfying the following property: for every positive number $\varepsilon$ such that $\varepsilon'_0\le \varepsilon\le \varepsilon_0$ and every positive integer $\ell\ge \ell'_0$,
\begin{equation} \label{eqn:enlarged ideal for lct 0}
\mathcal{J}_{\sigma}(X,\Delta,\lambda\|D\|)=\mathcal{J}\left(X,\Delta,\lambda_{\ell}(1+\varepsilon)\left\|D+\frac{1}{\ell!}A\right\|\right)
\end{equation}
where we denote $\lambda=\lct_{\sigma}(X,\Delta,D)$ and $\lambda_{\ell}=\lct_{\sigma}(X,\Delta,D+\tfrac{1}{\ell !}A)$. By the definition of $\mc{J}_{\sigma}$, there exists a positive number $\varepsilon_{0}>0$ such that
$$ \mc{J}_{\sigma}(X,\Delta,\lambda\norm{D})=\mc{J}_{-}(X,\Delta,(1+2\varepsilon_{0})\lambda\norm{D}). $$
By the definition of $\mc{J}_{-}$, we also have
$$ \mc{J}_{-}(X,\Delta,(1+\varepsilon)\lambda_{\ell}\norm{D})\subseteq \mc{J}(X,\Delta,(1+\varepsilon)\lambda_{\ell}\norm{D+\tfrac{1}{\ell !}A}) $$
for any $0<\varepsilon\le \varepsilon_{0}$ and any $\ell>0$. By Lemma \ref{lem:limit}, there exists a positive integer $\ell_{0}''=\ell_{0}''(\varepsilon_{0})$ such that for any $0<\varepsilon\le \varepsilon_{0}$ and $\ell\ge\ell_{0}''$,
$$ (1+\varepsilon)\lambda_{\ell}\le (1+2\varepsilon_{0})\lambda $$
holds. Thus, we have
\begin{align*}
\mc{J}_{\sigma}(X,\Delta,\lambda\norm{D})&=\mc{J}_{-}(X,\Delta,(1+2\varepsilon_{0})\lambda\norm{D})\\
&\subseteq \mc{J}_{-}(X,\Delta,(1+\varepsilon)\lambda_{\ell}\norm{D})\\
&\subseteq \mc{J}(X,\Delta,(1+\varepsilon)\lambda_{\ell}\norm{D+\tfrac{1}{\ell !}A})
\end{align*}
for any $0<\varepsilon\le \varepsilon_0$ and $\ell\ge\ell''_0$.

On the other hand, for any $0<\varepsilon'_0<\frac{\varepsilon_0}{2}$, there exists a positive integer $\ell_{0}'''=\ell_{0}'''(\varepsilon_{0}')$ such that for any $\ell\ge \ell_{0}'''$,
$$ \mc{J}_{-}(X,\Delta,(1+\varepsilon_{0}')\lambda\norm{D})=\mc{J}(X,\Delta,(1+\varepsilon_{0}')\lambda\norm{D+\tfrac{1}{\ell !}A}) $$
holds. Hence, for $\varepsilon\ge \varepsilon_{0}'$ and $\ell\ge \ell_{0}'=\max\{\ell_{0}'',\ell_{0}'''\}$ (which depends on both $\varepsilon_0$ and $\varepsilon_0'$), we have
\begin{align*}
\mc{J}_{\sigma}(X,\Delta,\lambda\norm{D})&=\mc{J}_{-}(X,\Delta,(1+2\varepsilon_{0})\lambda\norm{D})\\
&\subseteq \mc{J}(X,\Delta,(1+\varepsilon)\lambda_{\ell}\norm{D+\tfrac{1}{\ell !}A})\\
&\subseteq \mc{J}(X,\Delta,(1+\varepsilon_{0}')\lambda\norm{D+\tfrac{1}{\ell !}A})\\
&=\mc{J}_{-}(X,\Delta,(1+\varepsilon_{0}')\lambda\norm{D})\\
&=\mc{J}_{\sigma}(X,\Delta,\lambda\norm{D}).
\end{align*}

\medskip

\noindent\textbf{Step 2.}
We define the graded sequence of ideals $\mathfrak{a}_{m,\ell}$ as follows:
$$\mathfrak{a}_{m,\ell}\coloneqq \begin{cases}
    \mathfrak{b}\left(\left|mD+\frac{m}{\ell}A\right|\right) & \text{if }mD+\frac{m}{\ell}A\text{ is Cartier and }H^0\left(X,mD+\frac{m}{\ell}A\right)\ne 0, \\
    \phantom{.........}0 & \text{otherwise.}
\end{cases}$$
We claim that if $\ell\ge \ell'_0$, then
\begin{equation} \label{eqn:enlarged ideal for lct 5}
\lambda_{\ell}\le \lct\left(X_x,\Delta_x,\left(\mathfrak{a}_{\bullet,\ell!}\right)_x\right)\le (1+\varepsilon'_0)\lambda_{\ell}.
\end{equation}
The first inequality follows from Lemma \ref{lem:base change of multiplier ideal} (2). If $t\ge (1+\varepsilon'_0)\lambda_{\ell}$, then by (\ref{eqn:enlarged ideal for lct 0}) and Lemma \ref{lem:base change of multiplier ideal} (1), $\mathcal{J}\left(X_x,\Delta_x,\left(\mathfrak{a}_{\bullet,\ell!}\right)^t_x\right)\ne \mathcal{O}_{X,x}$. Therefore, Theorem \ref{thrm:sup and asymptotic} gives us the second inequality in (\ref{eqn:enlarged ideal for lct 5}).

\medskip

\noindent\textbf{Step 3.}
Let $n$ a positive integer such that $\mathfrak{m}^n_x\subseteq \mathcal{J}_{\sigma}(X,\Delta,\lambda\norm{D})_x$ holds. For simplicity, we denote $\lambda_{\ell,x}=\lct\left(X_x,\Delta_x,\left(\mathfrak{a}_{\bullet,\ell}\right)_x\right)$ and $\lambda'_{\ell,x}=\lct^{\mathfrak{m}^n_x}\left(X_x,\Delta_x,\left(\mathfrak{a}_{\bullet,\ell}\right)_x\right)$, respectively.

Since
$$ \lambda'_{\ell,x}=\lct^{\mathfrak{m}^n_x}\left(X_x,\Delta_x,\left(\mathfrak{a}_{\bullet,\ell!}\right)_x\right)=\inf\left\{\lambda\Bigm|\mathfrak{m}^n_x\nsubseteq\mathcal{J}\left(X_x,\Delta_x,\left(\mathfrak{a}_{\bullet,\ell!}\right)^{\lambda}_x\right)\right\}$$
holds by Proposition \ref{thrm:sup and asymptotic}, we see that by applying (\ref{eqn:enlarged ideal for lct 0}) and Lemma \ref{lem:base change of multiplier ideal} (1), we obtain
\begin{equation} \label{a}
\lambda'_{\ell,x}\ge (1+\varepsilon_0)\lambda_{\ell}.
\end{equation}
Moreover, (\ref{eqn:enlarged ideal for lct 5}) implies that
\begin{equation} \label{aa} \lambda_{\ell,x}\le (1+\varepsilon'_0)\lambda_{\ell}.
\end{equation}
Therefore, since $\varepsilon'_0\le \frac{\varepsilon_0}{2}$ holds, by combining (\ref{a}) and (\ref{aa}), we obtain
$$ \ceil{\frac{n}{\lambda'_{\ell,x}-\lambda_{\ell,x}}}+1\le \ceil{\frac{2n}{\varepsilon_0\lambda_{\ell}}}+1.$$

\vspace{1em}
By Lemma \ref{lem:limit}, there exists a positive integer $\ell''''_0$ such that
$$ \ceil{\frac{2n}{\varepsilon_0\lambda_{\ell}}}+1=\ceil{\frac{2n}{\varepsilon_0\lambda}}+1$$
holds for any $\ell\ge \ell''''_0$. Set $\ell_0\coloneqq \max\{\ell'_0,\ell''''_0\}$. Then for any $\ell\ge \ell_0$, we have
\begin{equation} \label{eqn:enlarged ideal for lct 9}
\ceil{\frac{n}{\lambda'_{\ell,x}-\lambda_{\ell,x}}}+1\le \ceil{\frac{2n}{\varepsilon_0\lambda}}+1.
\end{equation}

\medskip

\noindent\textbf{Step 4.} Recall that $n$ is a positive integer such that $\mathfrak{m}^n_x\subseteq \mathcal{J}_{\sigma}(X,\Delta,\lambda\|D\|)_x$ holds and let $\ell\ge \ell_0$. By (\ref{eqn:enlarged ideal for lct 9}), we have
$$
\ceil{\frac{n}{\lambda'_{\ell,x}-\lambda_{\ell,x}}}+1 \leq p
$$
where $p\coloneqq \ceil{\frac{2n}{\varepsilon_0\lambda}}+1$.
Let us define
$$ \mathfrak{c}_{m,\ell}\coloneqq \sum^m_{i=0}\mathfrak{a}_{i,\ell}\cdot \mathfrak{m}^{p(m-i)}_x.$$
We remark that (2) and (3) in Proposition \ref{prop:enlarged ideal for lct} follow from the definition of $\mathfrak{c}_{m,\ell}$.

Fix an integer $\ell\ge \ell_0$. Let $0<\varepsilon'<1$ be a positive number such that $p>\frac{n}{(1-\varepsilon')\lambda'_{\ell,x}-\lambda_{\ell,x}}$. Such a real number $\varepsilon'$ exists by the definition of $p$. Note that (3) of Proposition \ref{prop:enlarged ideal for lct} implies 
$$ \lct\left(X_x,\Delta_x,\left(\mathfrak{c}_{\bullet,\ell !}\right)_x\right)\le\inf_{\nu\in V_{\varepsilon'}}\frac{A_{X,\Delta}(\nu)}{\min\{\nu(\mathfrak{a}_{\bullet,\ell!}),p\cdot \nu(\mathfrak{m}_{x})\}}$$
where $V_{\varepsilon'}\coloneqq \left\{\nu\in \mathrm{Val}^*_{X,x}\Bigm| \frac{A_{X,\Delta}(\nu)}{\nu(\mathfrak{a}_{\bullet,\ell!})}\le \frac{\lambda_{\ell,x}}{1-\varepsilon'}\right\}$. Furthermore, the following inequalities hold:
$$
\frac{n\cdot \nu(\mathfrak{m}_{x})}{\nu(\mathfrak{a}_{\bullet,\ell!})}\ge \lambda'_{\ell,x}-\frac{A_{X,\Delta}(\nu)}{\nu(\mathfrak{a}_{\bullet,\ell!})}\ge \lambda'_{\ell,x}-\frac{\lambda_{\ell,x}}{1-\varepsilon'}
$$
for any $\nu\in V_{\varepsilon'}$. Therefore, we have
$$
\frac{1}{\lct(X,\Delta,\mathfrak{c}_{\bullet,\ell!})}\ge \sup_{\nu\in V_{\varepsilon'}}\frac{\nu(\mathfrak{a}_{\bullet,\ell!})}{A_{X,\Delta}(\nu)}\cdot\min\left\{1,\frac{p}{n}\left(\lambda'_{\ell,x}-\frac{\lambda_{\ell,x}}{1-\varepsilon'}\right)\right\}
=\sup_{\nu\in V_{\varepsilon'}}\frac{\nu(\mathfrak{a}_{\bullet,\ell!})}{A_{X,\Delta}(\nu)}
=\frac{1}{\lambda_{\ell,x}}
$$
which implies 
\begin{equation} \label{eqn:enlarged ideal for lct 8}
\lambda_{\ell,x}\ge \lct\left(X_x,\Delta_x,\left(\mathfrak{c}_{\bullet,\ell!}\right)_x\right).
\end{equation}

Thus, we obtain that
$$
\begin{aligned}
\lct(X,\Delta,\mathfrak{a}_{\bullet,\ell!})&\leq \lct\left(X,\Delta,\mathfrak{c}_{\bullet,\ell!}\right)
\\ &\leq\lct\left(X_x,\Delta_x,\left(\mathfrak{c}_{\bullet,\ell !}\right)_x\right) 
\\ &\leq \lct\left(X_x,\Delta_x,(\mathfrak{a}_{\bullet,\ell!})_x\right)
\\ &\leq(1+\varepsilon'_0)\cdot\lct\left(X,\Delta,\mathfrak{a}_{\bullet,\ell!}\right),
\end{aligned}
$$
where the first inequality follows from $\mathfrak{a}_m\subseteq \mathfrak{c}_m$, and the second inequality follows from Theorem \ref{thrm:sup and asymptotic} and Lemma \ref{lem:base change of multiplier ideal} (1). The third and fourth inequalities follow from (\ref{eqn:enlarged ideal for lct 8}) and (\ref{eqn:enlarged ideal for lct 5}), respectively. The above inequalities give us Proposition \ref{prop:enlarged ideal for lct} (1). 
\end{proof}

The role of Proposition \ref{prop:enlarged ideal for lct} is as follows. We first need to establish the uniform lower bound of $\nu_i(\mathfrak{m}_x)$ for minimizing valuations $\nu_i$. Note that in \cite{Xu2} such a lower bound is given as 
$$
\nu_i(\mathfrak{m}_x)\ge \frac{1}{2p\cdot \mathrm{lct}(X,\Delta,\mathfrak{a}_{\bullet})}.
$$
The inequalities in (2) of Proposition \ref{prop:enlarged ideal for lct} play the role of the lower bound.

\section{Proofs of main results}\label{sect:proofs}
We prove our main results in this section. Let us start by proving Theorem \ref{thrm:1}.

\begin{proof}[Proof of Theorem \ref{thrm:1}]

The proof consists of four steps. In the first step, we construct a graded sequence of ideals $\mathfrak{c}_{\bullet,\ell_j}$ that approximates the log canonical threshold $\lct_\sigma(X,\Delta,D)$.
In the second step, we find a convergent sequence of valuations $\{\nu_j\}$, each of which computes $\lct_\sigma(X,\Delta, \mathfrak{c}_{\bullet,\ell_j})$. In the latter two steps, we prove that the limit $\nu_0$ of the convergent sequence computes $\lct_{\sigma}(X,\Delta,D)$. By showing that there exists a quasi-monomial valuation $\omega$ which also computes $\lct_{\sigma}(X,\Delta,D)$, we finish the proof. 

We will consider only the case where $\lct_{\sigma}(X,\Delta,D)<\infty$. Indeed, if $\lct_{\sigma}(X,\Delta,D)=\infty$, then any $\omega\in \mathrm{Val}_X$ computes $\lct_{\sigma}(X,\Delta,D)$.

\smallskip

\noindent\textbf{Step 1.} We aim to construct a sequence of valuations $\nu'_{j}$ in $\textrm{Val}_X$ such that the centers $c_X(\nu'_j)$ all coincide to a fixed point of $X$ and each $\nu'_j$ computes the log canonical thresholds $\lct_\sigma(X,\Delta,\mathfrak{a}_{\bullet,\ell_j})$ for some subsequence of graded ideals $\{\mathfrak{a}_{m,j}\}$ which we define now. For an ample divisor $A$ on $X$, define $\mathfrak{a}_{m,\ell}$ as in the proof of Proposition \ref{prop:enlarged ideal for lct}, by
$$\mathfrak{a}_{m,\ell}\coloneqq \begin{cases}
    \mathfrak{b}\left(\left|mD+\frac{m}{\ell}A\right|\right) & \text{if }mD+\frac{m}{\ell}A\text{ is Cartier and }H^0\left(X,mD+\frac{m}{\ell}A\right)\ne 0, \\
    \phantom{.........}0 & \text{otherwise.}
\end{cases}$$
Then, by Proposition \ref{prop:enlarged ideal for lct}, there exist a positive number $\varepsilon_0>0$, a positive integer $p$, and a graded sequence of ideals $\mathfrak{c}_{\bullet,\ell}$ in $\mathcal{O}_X$ such that for any $0<\varepsilon'_0<\frac{\varepsilon_0}{2}$, there is an increasing sequence $\{\ell_j=\ell_j(\varepsilon'_0)\}_{j\ge 0}$ that satisfies the following properties: let $x$ be the generic point of an irreducible component of $\mathcal{Z}(\mathcal{J}_{\sigma}(X,\Delta,\lct_{\sigma}(X,\Delta,D)\|D\|))$. Then
\begin{enumerate}[leftmargin=8mm]
\item[(i)] we have
$$ \lct_{\sigma}\left(X,\Delta,D+\frac{1}{\ell_j}A\right)\le \lct\left(X_x,\Delta_x,\left(\mathfrak{c}_{\bullet,\ell_j}\right)_x\right)\le (1+\varepsilon'_0)\cdot\lct_{\sigma}\left(X,\Delta,D+\frac{1}{\ell_j} A\right),$$
and
\item[(ii)] there is a positive integer $p$ (which is independent of $\varepsilon'_0$ and $j$) such that $\mathfrak{m}^p_x\subseteq \mathfrak{c}_{1,\ell_j}$.
\end{enumerate}

Note that by (ii), $\mathfrak{m}^{mp}_x\subseteq \mathfrak{c}_{m,\ell_j}$ for any sufficiently divisible positive integer $m$. Therefore,
$$ \nu(\mathfrak{m}_x)\ge \frac{1}{mp}\nu(\mathfrak{c}_{m,\ell_j})$$
which implies that for any $j$ and any $\nu\in \mathrm{Val}_{X,x}$, the following inequality holds:
\begin{equation} \label{eqn:lowerbound}
\nu(\mathfrak{m}_x)\ge \frac{1}{p}\nu(\mathfrak{c}_{\bullet,\ell_j}).
\end{equation}

Let $\delta>0$ be a positive number such that $\lct\left(X_x,\Delta_x,\left(\mathfrak{c}_{\bullet,\ell_0}\right)_x\right)<\frac{1}{p\delta}$.
Fix a subset $W\coloneqq \left\{\nu\in \mathrm{Val}_{X,x}\mid \nu(\mathfrak{m}_x)\ge \delta\text{ and }A_{X,\Delta}(\nu)\le 1 \right\}$.
By the condition on $\delta$, it is enough to consider the valuations $\nu$ such that $\frac{A_{X,\Delta}(\nu)}{\nu(\mathfrak{c}_{\bullet.\ell_j})}<\frac{1}{p\delta}$ to compute $\lct\left(X_x,\Delta_x,\left(\mathfrak{c}_{\bullet,\ell_0}\right)_x\right)$. Note that we may also assume that $A_{X,\Delta}(\nu)\leq 1$. Then by (\ref{eqn:lowerbound}), we have
$$
\nu(\mathfrak{m}_x)\ge \frac{1}{p}\nu\left(\mathfrak{c}_{\bullet,\ell_j}\right)\ge \delta.
$$
Thus, to compute $\lct\left(X_x,\Delta_x,\left(\mathfrak{c}_{\bullet,\ell_j}\right)_x\right)$, it is enough to consider the valuations in $W$. Therefore, we have the following equality
\begin{align*}
\lct\left(X_x,\Delta_x,\left(\mathfrak{c}_{\bullet,\ell_j}\right)_x\right)=\inf_{\nu\in W}\frac{A_{X,\Delta}(\nu)}{\nu\left(\mathfrak{c}_{\bullet,\ell_j}\right)}.
\end{align*}

Now, we define a function 
$$ \phi_{j}(\nu)\coloneqq \frac{A_{X,\Delta}(\nu)}{\nu\left(\mathfrak{c}_{\bullet,\ell_j}\right)}$$
for any $\nu\in W$. As $A_{X,\Delta}(-)\colon W\to \R$ is lower semi-continuous by \cite[Section 7.2]{Blu} and $\nu(\mathfrak{c}_{\bullet,\ell_j})$ is upper semi-continuous by \cite[Lemma 6.1]{JM}, the function $\phi_{j}$ is lower semi-continuous. Therefore, there exists a valuation $\nu'_j\in W$ which computes $\lct\left(X_x,\Delta_x,\left(\mathfrak{c}_{\bullet,\ell_j}\right)_x\right).$
Note that we may assume $A_{X,\Delta}(\nu'_{j})=1$ by rescaling $\nu'_j$.
\medskip

\noindent\textbf{Step 2.} 
By Remark \ref{rmk:xu main result}, there exist quasi-monomial valuations $\nu_{j}\in \mathrm{Val}_X$ which compute $\lct\left(X_x,\Delta_x,\left(\mathfrak{c}_{\bullet,\ell_j}\right)_x\right)$ and satisfy the following 
\begin{itemize}
\item[(1)] $A_{X,\Delta}(\nu_{j})=A_{X,\Delta}(\nu'_{j})=1$
\item[(2)] $\nu_j\ge \nu'_j$, and
\item[(3)] $c_X(\nu_{j})=x$.
\end{itemize}

We see that $\nu_j\in W$ by (1), (2), and (3). Since $W$ is sequentially compact (cf. \cite[Proposition 3.9]{LX2}), after passing to a subsequence of $\{\nu_j\}_{j\ge 1}$, the sequence $\{\nu_j\}_{j\ge 1}$ converges to some $\nu_0\in W$.

\smallskip

\noindent\textbf{Step 3.} This step is devoted to proving that $\nu_0$ computes $\lct_{\sigma}(X,\Delta,D)$.

For any $\varepsilon>0$, there exists a positive integer  $j_0$ such that for $j>j_0$, we have the following inequalities:
\begin{align*}
    \frac{A_{X,\Delta}(\nu)}{\sigma_{\nu}\left(D+\frac{1}{\ell_j}A\right)}&\ge \lct_{\sigma}\left(X,\Delta,D+\frac{1}{\ell_j}A\right) & (1)
     \\ 
     & \ge \frac{1}{1+\varepsilon'_0}\cdot\lct\left(X_x,\Delta_x,\left(\mathfrak{c}_{\bullet,\ell_j}\right)_x\right) & (2)
     \\ & = \frac{1}{1+\varepsilon'_0}\cdot\frac{1}{\nu_j(\mathfrak{c}_{\bullet,\ell_j})}   & (3)  \\ 
      &\ge \frac{1}{1+\varepsilon'_0}\cdot\frac{A_{X,\Delta}(\nu_0)}{\nu_j(\mathfrak{c}_{\bullet,\ell_j})} & (4)
\\
&\ge \frac{1}{1+\varepsilon'_0}\cdot\frac{A_{X,\Delta}(\nu_0)}{\sigma_{\nu_j}\left(D+\frac{1}{\ell_j}A\right)} & (5)
    \\ &\ge \frac{1}{1+\varepsilon'_0}\cdot\frac{A_{X,\Delta}(\nu_0)}{\sigma_{\nu_j}\left(D\right)} & (6)
    \\ &=\frac{1}{1+\varepsilon'_0}\cdot\frac{A_{X,\Delta}(\nu_0)}{\sigma_{\nu_0}(D)}-\varepsilon, & (7)
\end{align*}

where

\begin{itemize}
    \item[(1)] follows from the definition of log canonical threshold, 
    \item[(2)] is due to Proposition \ref{prop:enlarged ideal for lct} (1) ((i) in Step 1), 
    \item[(3)] holds since $\nu_j$ computes $\lct\left(X_x,\Delta_x,\left(\mathfrak{c}_{\bullet,\ell_j}\right)_x\right)$ (Step 2) and $A_{X,\Delta}(\nu_j)=1$ ((1) in Step 2), 
    \item[(4)] follows from the lower semi-continuity of $A_{X,\Delta}(-)$ (cf. \cite[Section 7.2]{Blu}) and 
$$ A_{X,\Delta}(\nu_0)\le \liminf_{j\to \infty}A_{X,\Delta}(\nu_j)=1,$$
\item[(5)] follows from (3) of Proposition \ref{prop:enlarged ideal for lct},
\item[(6)] is due to $\sigma_{\nu}(D)\ge \sigma_{\nu}(D+A')$ for any $\nu\in \mathrm{Val}_X$ and any ample $\Q$-divisor $A'$ on $X$, and
\item[(7)] follows from Proposition \ref{prop:continuity}.
\end{itemize}

Finally, by letting $j\to \infty$, $\varepsilon'_0\to 0$, and $\varepsilon\to 0$ successively, we see that $\nu_{0}$ computes $\lct_{\sigma}(X,\Delta,D)$.

\medskip

\noindent\textbf{Step 4}. Finally, we finish the proof of Theorem \ref{thrm:1} by proving that there is a quasi-monomial valuation $\omega$ which computes $\lct_{\sigma}(X,\Delta,D)$. This step is similar to the proof of \cite[Lemma 4.7]{BJ}.

Let us begin by proving that $\nu_0$ computes $\lct(X,\Delta,\mathfrak{a}_{\bullet}(\nu_0))$.
We first claim the following inequality
$$
\frac{A_{X,\Delta}(\nu_{0})}{\nu_0(\mathfrak{a}_{\bullet}(\nu_0))}\le \frac{A_{X,\Delta}(\nu)}{\nu(\mathfrak{a}_{\bullet}(\nu_0))}
$$
for any $\nu\in\Val_X$.
By scaling, it suffices to check the case $\nu(\mathfrak{a}_{\bullet}(\nu_0))=1$.
In this case, $\nu(\mathfrak{a}_p(\nu_0))\ge p$ for any nonnegative integer $p$. Therefore, we have $\nu\ge \nu_0$, and in turn, we obtain the inequality $\sigma_{\nu}(D)\ge \sigma_{\nu_0}(D)$. Since $\nu_0$ computes $\lct_{\sigma}(X,\Delta,D)$, we obtain
$$ \frac{A_{X,\Delta}(\nu_{0})}{\sigma_{\nu_0}(D)}\le \frac{A_{X,\Delta}(\nu)}{\sigma_{\nu}(D)}$$
and $A_{X,\Delta}(\nu_{0})\le A_{X,\Delta}(\nu)$ for any $\nu\in \mathrm{Val}_X$ with $\nu(\mathfrak{a}_{\bullet}(\nu_0))=1$. Note also that $\nu_0(\mathfrak{a}_{\bullet}(\nu_0))=1$. Therefore, we obtain the following inequality
$$
\frac{A_{X,\Delta}(\nu_{0})}{\nu_0(\mathfrak{a}_{\bullet}(\nu_0))}=A_{X,\Delta}(\nu_{0})\le A_{X,\Delta}(\nu)=\frac{A_{X,\Delta}(\nu)}{\nu(\mathfrak{a}_{\bullet}(\nu_0))}
$$
which implies the claim.

By Remark \ref{rmk:xu main result}, there is a quasi-monomial valuation $\omega\ge \nu_0$ such that $A_{X,\Delta}(\omega)=A_{X,\Delta}(\nu_{0})$. Moreover, by the inequality $\sigma_{\omega}(D)\ge \sigma_{\nu_0}(D)$, we have
$$ \frac{A_{X,\Delta}(\omega)}{\sigma_{\omega}(D)}=\frac{A_{X,\Delta}(\nu_{0})}{\sigma_{\omega}(D)}\le \frac{A_{X,\Delta}(\nu_{0})}{\sigma_{\nu_0}(D)}=\lct_{\sigma}(X,\Delta,D).$$
By the definition of the log canonical threshold, we conclude that $\omega$ also computes $\lct_{\sigma}(X,\Delta,D)$.
\end{proof}

For a pklt triple $(X,\Delta,D)$, Proposition \ref{prop:perturbation} shows that $(X,\Delta,(1+\varepsilon)D)$ is also pklt.

\begin{proposition} \label{prop:perturbation}
Let $(X,\Delta,D)$ be a pklt triple. Then for sufficiently small $\varepsilon>0$, $(X,\Delta,(1+\varepsilon)D)$ is pklt.
\end{proposition}

\begin{proof}
By Theorem \ref{thrm:1}, there exists a quasi-monomial valuation $\nu$ which computes $\lct_{\sigma}(X,\Delta,D)$.
Let $(Y,E)$ be a log-smooth model over $(X,\Delta)$ such that $\nu\in \QM_{\eta}(Y,E)$ where $E=E_{1}+\cdots+E_{r}$ is a reduced simple normal crossing divisor and $\eta=c_{Y}(\nu)$ is the generic point of a connected component of $E_{1}\cap \cdots \cap E_{i}\neq \varnothing$ for some $i$. Let $\|\cdot\|$ be a metric on $\mathrm{QM}_{\eta}(Y,E)$ that induces the topology on $\mathrm{QM}_{\eta}(Y,E)$.

Let $\phi(\nu)=A_{X,\Delta}(\nu)-\sigma_{\nu}(D)$. Then clearly $\phi$ is homogeneous of degree 1. Note that $A_{X,\Delta}(\nu)$ is linear on $\QM_{\eta}(Y,E)$ and $\sigma_{\nu}(D)$ is concave on $\QM_{\eta}(Y,E)$.  (cf. \cite[Lemma 2.9]{XuZhuang}.) Therefore, $\phi$ is convex on $\QM_{\eta}(Y,E)$. Note that we may assume that $\nu$ is in the interior of $\mathrm{QM}_{\eta}(Y,E)$. Since every convex function is locally Lipschitz on an open and convex subset of a Euclidean space, there exist positive real numbers $C$ and $\delta$ such that
$$ |\phi(\nu)-\phi(\omega)|<C\|\nu-\omega\|$$
for all valuations $\omega\in \QM_{\eta}(Y,E)$ with $\|\nu-\omega\|\le \delta$. On the other hand, by \cite[Lemma 2.7]{LX1}, for each $t>0$, there exist divisorial valuation $\omega_{t}\in \QM_{\eta}(Y,E)$ and a positive rational number $q$ such that
\begin{enumerate}[label=$\bullet$]
\item $q\cdot \omega_{t}=\ord_{F_t}$, where $F_t$ is a prime divisor over $X$,
\item $\|\nu-\omega_{t}\|<\frac{t}{q}$.
\end{enumerate}

Since $(X,\Delta,D)$ is pklt, there is a positive real number $\varepsilon>0$ such that $A_{X,\Delta}(E)-\sigma_{E}(D)\ge \varepsilon$ for all prime divisors $E$ over $X$. By taking a sufficiently small $t>0$, we have
$$ \phi(q\nu)\ge \phi(q\omega_{t})-|\phi(q\nu)-\phi(q\omega_{t})|>\varepsilon-Ct>0, $$
which implies that $A_{X,\Delta}(\nu)-\sigma_{\nu}(D)>0$. Moreover, we obtain that
$$ \frac{A_{X,\Delta}(E)}{\sigma_{E}(D)}\ge \frac{A_{X,\Delta}(\nu)}{\sigma_{\nu}(D)}>1 $$
for every prime divisor $E$ over $X$. Now, the assertion follows from Lemma \ref{lem:pklt to plc}.
\end{proof}

By Proposition \ref{prop:perturbation}, we provide a proof of Theorem \ref{thrm:2}.
\begin{proof}[Proof of Theorem \ref{thrm:2}]
The proof is similar to the proof of \cite[Proposition 4.1]{LP}. If $K_X+\Delta+(1+\varepsilon)D$ is nef for any $0<\varepsilon\ll 1$, then there is nothing further to do.

Let $B$ be a big and effective divisor on $X$, and $\varepsilon_0$ a positive rational number which satisfies the following properties for any $0< \varepsilon<\varepsilon_0$.
	\begin{enumerate}[(1)]
		\item $K_X+\Delta+(1+\varepsilon)D+B$\text{ is nef, and}
		\item for all rational numbers $0<\varepsilon'<1$, there is an effective $\mathbb{Q}$-divisor $\Delta_{\varepsilon,\varepsilon'}$ such that $\Delta_{\varepsilon,\varepsilon'}\sim_{\Q}(1+\varepsilon)D+\varepsilon'B$ and 
      $(X,\Delta+\Delta_{\varepsilon,\varepsilon'}+(1-\varepsilon')B)$ is klt.
	\end{enumerate}
Such a divisor $B$ and a positive number $\varepsilon_0>0$ exist. Indeed, let $B$ be a sufficiently ample divisor. By Proposition \ref{prop:perturbation}, $(X,\Delta,(1+\varepsilon_0)D)$
is pklt for some sufficiently small $\varepsilon_0>0$. If $0< \varepsilon<\varepsilon_0$, then by Lemma \ref{lem:pklt to plc} the triple $(X,\Delta,(1+\varepsilon)D+\varepsilon'B)$ is pklt for any $0<\varepsilon'<1$. By Theorem \ref{thm:compl}, the pair $(X,\Delta+\Delta_{\varepsilon,\varepsilon'}+(1-\varepsilon')B)$ is klt.

Let us define $ \lambda_1\coloneqq \inf\{t\ge 0 \mid K_X+\Delta+(1+\varepsilon)D+tB\text{ is nef}\}$.
By \cite[Lemma 2.2]{KMM}, there is a $(K_X+\Delta+\Delta_{\varepsilon,\varepsilon'})$-negative extremal ray $R\subseteq \overline{\mathrm{NE}}(X)$ such that
$$ (K_X+\Delta+\Delta_{\varepsilon,\varepsilon'}+(\lambda_1-\varepsilon')B)\cdot R=(K_X+\Delta+(1+\varepsilon)D+\lambda_1B)\cdot R=0$$
for $0<\varepsilon'<\lambda_1$. We see that $B\cdot R>0$, and thus $R$ is $(K_X+\Delta+(1+\varepsilon)D)$-negative. By contracting $R$, we obtain either a divisorial contraction or a flip by \cite{BCHM}
$$ (X_0,\Delta_0)\coloneqq  (X,\Delta)\overset{f_1}{\dashrightarrow} (X_1,\Delta_1).$$
By defining $\Delta_1\coloneqq f_{1*}\Delta$, $D_1\coloneqq f_{1*}D$, $B_1\coloneqq  \lambda_1 f_{1*}B$ and $\Delta_{\varepsilon,\varepsilon',1}\coloneqq f_{1*}\Delta_{\varepsilon,\frac{\varepsilon'}{\lambda_1}}$, we can repeat the above process, and we complete the proof. Note that $(X_1,\Delta_1,(1+\varepsilon)D_1)$ is pklt by Proposition \ref{prop:pot discr}. 
\end{proof}

\begin{proof}[Proof of Corollary \ref{cor:1}]
This is immediate if we let $D=-(K_X+\Delta)$ in Theorem \ref{thrm:2}.
\end{proof}

\begin{proof}[Proof of Corollary \ref{cor:2}]
Let $\rho(X)$ be the Picard number of $X$. Let
$$ (X_0,\Delta_0)\coloneqq (X,\Delta)\overset{f_1}{\dashrightarrow} (X_1,\Delta_1)\overset{f_2}{\dashrightarrow} \cdots$$
be a sequence of steps of $-(K_X+\Delta)$-MMP. Note that the birational maps $f_1,f_2,\cdots$ are divisorial contractions, and therefore $\rho(X_i)=\rho(X)-i$. If the MMP does not terminate, then $\rho(X_i)<0$ for some $i$, which is a contradiction.
\end{proof}

\begin{remark}
Note that Corollaries \ref{cor:1} and \ref{cor:2} are valid only for pklt pairs $(X,\Delta)$ with $\Q$-divisor $\Delta$ while the general theory of pklt pairs can be developed with $\R$-divisor $\Delta$ (cf. \cite{CP}). Our proofs are not applicable to the $\R$-divisor case because the proofs rely on the results on the asymptotic multiplier ideals that are only defined for $\Q$-divisors.
 \end{remark}

\end{document}